%% file: 2012ASEPAskey.tex
\documentclass[12pt]{amsart} 
\title[Combinatorics of the asymmetric exclusion process]{Tableaux combinatorics for the asymmetric exclusion process and Askey-Wilson polynomials}
\author{Sylvie Corteel and Lauren K. Williams}
\thanks{Both authors were partially supported by 
  the grants ANR blanc Gamma and ANR-08-JCJC-0011, and the second author
  was partially supported by the NSF grant DMS-0854432 and an Alfred Sloan Fellowship.}
\address{Laboratoire d'Informatique Algorithmique: Fondements et Applications,
Centre National de la Recherche Scientifique et Universit\'e Paris Diderot,
Paris 7, Case 7014, 75205 Paris Cedex 13
France} 
\email{corteel@liafa.jussieu.fr}
\address{Department of Mathematics, University of California, Berkeley,
Evans Hall Room 913, Berkeley, CA 94720}
\email{williams@math.berkeley.edu}



\usepackage{pstricks, pst-node, amssymb}
\usepackage{bbm}

\psset{unit=1pt, arrowsize=4pt, linewidth=.7pt}
\psset{linecolor=blue}
\newgray{grayish}{.90}
\newrgbcolor{embgreen}{0 .5 0}

\def\vblack(#1, #2)#3{\cnode*[linecolor=black](#1, #2){3}{#3}}
\def\vwhite(#1,#2)#3{\cnode[linecolor=black,fillcolor=white,fillstyle=solid](#1,
#2){3}{#3}}
\countdef\x=23
\countdef\y=24
\countdef\z=25
\countdef\t=26

\def\tbox(#1,#2)#3{
\x=#1 \y=#2
\multiply\x by 12
\multiply\y by 12
\z=\x \t=\y
\advance\z by 12
\advance\t by 12
\psline(\x,\y)(\x,\t)(\z,\t)(\z,\y)(\x,\y)
\advance\x by 6
\advance\y by 6
\rput(\x,\y){{\bf #3}}}

\usepackage{amsmath, amsthm, amssymb, amsbsy}
\usepackage{amsfonts, latexsym, stmaryrd, amscd, xy}
\usepackage[mathscr]{eucal}
\usepackage{epsfig}

\newtheorem{theorem}{Theorem}[section]

\newtheorem{proposition}[theorem]{Proposition}
\newtheorem{lemma}[theorem]{Lemma}
\newtheorem{observation}[theorem]{Observation}
\newtheorem{example}[theorem]{Example}

\newtheorem{remark}[theorem]{Remark}

\newtheorem{problem}[theorem]{Problem}

\newtheorem{definition}[theorem]{Definition}

\usepackage{amsfonts}
\usepackage{xy}
\usepackage{amssymb}
\usepackage{amsmath}
\newcommand{\T}{\mathcal{T}}
\newcommand{\ttt}{\tau}

\newcommand{\Z}{\mathbb Z}

\newcommand{\indicator}{\mathbbm{1}}

\DeclareMathOperator{\wt}{wt}

\newcommand{\thmrefer}[1]{\renewcommand\thetheorem
  {\protect\ref{#1}}\addtocounter{theorem}{-1}}

\xyoption{all}
\CompileMatrices

\begin{document}



\begin{abstract}

Introduced in the late 1960's \cite{bio, Spitzer}, 
the asymmetric exclusion process (ASEP) is an important model
from statistical mechanics which describes a system of interacting 
particles hopping left and right on a one-dimensional lattice of $n$
sites.  It has been cited as a model for traffic flow and 
protein synthesis.  In the most general form of the ASEP with open boundaries,
particles may enter and exit at the left with probabilities 
$\alpha$ and $\gamma$, and they may exit and enter at the right
with probabilities $\beta$ and $\delta$.  In the bulk, 
the probability of hopping left is $q$ times the probability of hopping
right.
The first main result of this paper is a combinatorial formula
for the stationary distribution of the ASEP with all parameters
general, in terms of 
a new class of tableaux which we call {\it staircase tableaux}.
This generalizes 
our previous work \cite{CW1, CW2} for the ASEP with 
parameters $\gamma=\delta=0$.
Using our first result and also results of Uchiyama-Sasamoto-Wadati 
\cite{USW}, 
we derive our second main result: a combinatorial 
formula for the moments of Askey-Wilson polynomials.
Since the early 1980's there has been a great deal of work
giving combinatorial formulas for moments of various other
classical orthogonal polynomials (e.g. Hermite, Charlier, Laguerre, Meixner).  However,
this is the first such formula for the 
Askey-Wilson polynomials, which are at the top of the hierarchy of
classical orthogonal polynomials.
\end{abstract}

\maketitle

\setcounter{tocdepth}{1}
\tableofcontents

\section{Introduction}

The asymmetric exclusion process (ASEP)
is an important model from statistical mechanics which
was introduced independently in the context of biology
\cite{bio} and in mathematics \cite{Spitzer} around 1970.
Since then there has been a huge amount of activity
on the ASEP and its variants for a number of reasons:
although the definition of the model is quite simple,
the ASEP is surprisingly rich.  For example, it exhibits
boundary-induced phase transitions, spontaneous symmetry
breaking, and phase separation.  Furthermore,
the ASEP is regarded as a primitive
model for 
translation in protein synthesis \cite{bio},
traffic flow \cite{traffic},
and formation of shocks \cite{shock};
it also appears in a kind of sequence alignment problem in
computational biology \cite{comput}.
Important mathematical techniques used to understand
this model include the Matrix Ansatz of Derrida, Evans, Hakim, and Pasquier
\cite{Derrida1}, and also the Bethe Ansatz
\cite{DE,GS, S}.

This paper concerns the ASEP on a one-dimensional
lattice of $n$ sites with open boundaries.
Particles may enter from the left at rate $\alpha dt$ and
from the right at rate $\delta dt$; they may 
exit the system to the right
at rate $\beta dt$
and to the left at rate $\gamma dt$.
The probability of hopping left and right is $qdt$ and $udt$, respectively.\footnote{Actually there is no loss of generality in setting $u=1$, so
we will often do so.}

Dating back to at least 1982, it was realized
that there were connections between this
model and combinatorics --
for example, Young diagrams appeared in \cite{SZ}, and 
Catalan numbers arose in 
the analyses of the stationary distribution of the 
ASEP in \cite{Derrida0, Derrida1}. 
More recently, beginning with papers of Brak and Essam \cite{BE},
and Duchi and Schaeffer \cite{jumping}, and followed by 
a number of works including
\cite{Angel,  BCEPR, Corteel, CW1, CW2, Viennot}, there has been a push to 
understand the roots of the combinatorial features of the ASEP.
Ideally, the goal is to find a combinatorial description of the
stationary distribution: that is, to express each component
of the stationary distribution
as a generating function for a set of combinatorial objects.
Up to now, the best result in this direction was provided by 
\cite{CW1, CW2}, which addressed this question 
when $\gamma=\delta=0$ 
 in terms of permutation tableaux,
combinatorial objects introduced in the context of total positivity on the 
Grassmannian \cite{Postnikov}.

There is a second reason why combinatorialists became intrigued by 
the ASEP.  Papers of Sasamoto \cite{Sasamoto} and subsequently
Uchiyama, Sasamoto, and Wadati \cite{USW} linked the ASEP with open
boundaries to orthogonal polynomials, in particular,
to the Askey-Wilson polynomials.
The Askey-Wilson polynomials are orthogonal polynomials
that sit at the top of the hierarchy of (basic) hypergeometric orthogonal
polynomials \cite{AW,GR,Koekoek}, in the sense that all other polynomials
in this hierarchy are limiting cases or specializations of the Askey-Wilson polynomials.
It is known that orthogonal polynomials have many combinatorial features:
indeed, starting around the early 1980's, mathematicians
including 
Flajolet \cite{Flajolet}, Viennot \cite{Viennot-book}, and
Foata \cite{Foata},
initiated a combinatorial
approach to orthogonal polynomials.  Since then, combinatorial formulas
have been given for the moments of (the weight functions
of) many of the polynomials in the Askey scheme, 
including
$q$-Hermite, Chebyshev, $q$-Laguerre, Charlier, Meixner, and Al-Salam-Chihara
polynomials, see e.g. \cite{IS, ISV, KSZ1,  MSW, SS}.
However, no such formula was known for the moments of the
Askey-Wilson polynomials.  Therefore when the paper \cite{USW} provided 
a close link between the moments of the Askey-Wilson polynomials
and the stationary distribution of the ASEP, this gave the idea 
that a complete combinatorial understanding of the ASEP  might
also give rise to the combinatorics of the Askey-Wilson moments.

In this paper we give a complete solution to both of the 
above problems.  Namely, we introduce some new combinatorial objects,
which we call
{\it staircase tableaux}, and we prove that 
generating functions for staircase tableaux describe
the stationary distribution
of the ASEP, with all parameters general.
We then use this result together with \cite{USW} to give
a combinatorial formula for
the moments of the Askey-Wilson polynomials.

The method of proof for our stationary distribution result
builds upon important work of 
Derrida, Evans, Hakim, and Pasquier \cite{Derrida1}, who
introduced a {\it Matrix Ansatz} as a tool for understanding
its stationary distribution.  Briefly, the Matrix Ansatz
says that if one can find matrices and vectors satisfying
certain relations (the {\it DEHP algebra}), then 
each component of the stationary distribution of the ASEP can be expressed
in terms of certain products of these matrices and vectors.
Knowing this Ansatz, the strategy for proving our stationary distribution
result which one would like to employ is the following:
find matrices and vectors satisfying the DEHP algebra,
and show that appropriate products enumerate staircase tableaux.

However, it has been known since 1996 \cite[Section IV]{ER} that when 
the parameters of 
the ASEP satisfy certain relations 
(for example $\alpha \beta = q^i \gamma \delta$), there is no 
representation of the DEHP algebra.  Therefore the above strategy 
cannot succeed when all parameters
of the ASEP are general.  What we do instead is to
introduce a slight generalization
of the Matrix Ansatz (see Theorem \ref{ansatz2}), which 
is more flexible, albeit harder to use: instead of checking three identities,
one must check three infinite families of identites.  
So our strategy is to find matrices and vectors 
such that appropriate products enumerate staircase tableaux,
and prove that they satisfy the relations of the Generalized Matrix Ansatz.
This second step is quite involved, as our ``matrices" and ``vectors"
are somewhat complicated (they have four and two indices each, see
Section \ref{def-matrices}), and there is no easy way to use induction
to prove the three infinite families of identities.

We believe that our new staircase tableaux deserve further study,
because of their combinatorial interest and their
potential connection to geometry.
For example, staircase tableaux of size $n$ have cardinality $4^n n!$,
and hence are in bijection with doubly-signed permutations.  In
\cite{CSW}, we prove this with an explicit bijection,
and also develop connections to other combinatorial objects, 
including matchings, permutations, and trees.
Furthermore, because of the connection to the ASEP, we know that our
staircase tableaux have some hidden symmetries which are not at all
apparent from their definition.
For instance, it is clear from the definition of the ASEP
that the model remains unchanged if we reflect it over the
$y$-axis, and exchange parameters $\alpha$ and $\delta$,
$\beta$ and $\gamma$, and $q$ and $u$.  However,
the corresponding bijection on the level of tableaux has so far eluded us.
Finally, staircase tableaux generalize permutation tableaux,
which index certain cells in the non-negative part of the
Grassmannian \cite{Postnikov}; it would be interesting to
better understand the relationship between
the tableaux, the ASEP, and the geometry,
and potentially generalize it to staircase tableaux.

It is worth mentioning that in recent years there has been an explosion
of activity \cite{BFP, BC, FS, BF, QV, BS, TW, TW2, TW3}
surrounding another version of the ASEP, in which particles hop
not on a finite lattice but on $\Z$.
Much of this interest was inspired by 
Johansson's work
\cite{Johansson}, which reinterprets the TASEP (totally asymmetric
exclusion process) as a
randomly growing Young diagram.  Johansson then used
the well-understood combinatorics of Young diagrams and semi-standard
tableaux to approach the problem of current fluctuations
in the TASEP.  In light of this, one may hope that a better understanding of
the combinatorics of staircase tableaux could
lead to even more results on the ASEP with open boundaries.

The structure of this paper is as follows.
In Section \ref{setup} we define the ASEP (with open boundaries), and in
Sections \ref{Results} and \ref{AW Section} we state our
main results on the ASEP and on Askey-Wilson polynomials.
In Section \ref{newansatz} we prove a generalized Matrix Ansatz,
and in Sections \ref{proof1} and \ref{AWproof} we prove
our results on the ASEP and Askey-Wilson polynomials.
Section \ref{conc} gives  open problems, and the Appendix describes 
how staircase tableaux generalize permutation tableaux and 
alternative tableaux \cite{Viennot}.

\textsc{Acknowledgments:} We would like to thank
Ira Gessel, Richard Stanley, and Jan de Gier
for interesting remarks.  We are also grateful to the referees for 
their thoughtful comments, which helped us to greatly improve the exposition.
Finally, we are grateful to Svante Janson, who 
found a gap in the proof of Proposition 6.11 in the published version
of this article; the problem was that to deduce (17) from (16), 
the old proof used identity $2E_{j,m+1}^n$ rather than identity 
$2E_{j,m}^n$,
as originally stated.    This version of the paper corrects the error, 
by giving a direct proof of Theorem 6.9 (2) (using the number of the 
published version).  We are also grateful to Pawel Hitczenko for
useful comments.

\section{The asymmetric exclusion process (ASEP)}\label{setup}

The ASEP is often defined using a
continuous time parameter \cite{Derrida1}.  
However, one can also define it 
as a discrete-time Markov chain, as we will do 
below.  The continuous and discrete-time definitions are equivalent
in the sense that their stationary distributions are the same
\cite{jumping}.

\begin{definition}
Let $\alpha$, $\beta$, $\gamma$, $\delta$,  $q$, and $u$ be constants such that 
$0 \leq \alpha \leq 1$, $0 \leq \beta \leq 1$, 
$0 \leq \gamma \leq 1$, $0 \leq \delta \leq 1$, 
$0 \leq q \leq 1$,
and $0 \leq u \leq 1$.
Let $B_n$ be the set of all $2^n$ words in the
language $\{\circ, \bullet\}^*$.
The ASEP is the Markov chain on $B_n$ with
transition probabilities:
\begin{itemize}
\item  If $X = A\bullet \circ B$ and
$Y = A \circ \bullet B$ then
$P_{X,Y} = \frac{u}{n+1}$ (particle hops right) and
$P_{Y,X} = \frac{q}{n+1}$ (particle hops left).
\item  If $X = \circ B$ and $Y = \bullet B$
then $P_{X,Y} = \frac{\alpha}{n+1}$ (particle enters from the left).
\item  If $X = B \bullet$ and $Y = B \circ$
then $P_{X,Y} = \frac{\beta}{n+1}$ (particle exits to the right).
\item  If $X = \bullet B$ and $Y = \circ B$
then $P_{X,Y} = \frac{\gamma}{n+1}$ (particle exits to the left).
\item  If $X = B \circ$ and $Y = B \bullet$
then $P_{X,Y} = \frac{\delta}{n+1}$ (particle enters from  the right).
\item  Otherwise $P_{X,Y} = 0$ for $Y \neq X$
and $P_{X,X} = 1 - \sum_{X \neq Y} P_{X,Y}$.
\end{itemize}
\end{definition}

Note that we will sometimes denote a state of the ASEP as a
word in $\{0,1\}^n$ and sometimes as a word in $\{\circ,\bullet\}^n$.
In these notations, the symbols $1$ and $\bullet$ denote a particle,
while $0$ and $\circ$ denote the absence of
a particle, which one can also think of as a white particle.

See Figure \ref{states} for an
illustration of the four states, with transition probabilities,
for the case $n=2$.  The probabilities on the loops
are determined by the fact that the sum of the probabilities
on all outgoing arrows from a given state must be $1$.

\begin{figure}[h]
\centering
\includegraphics[height=1.5in]{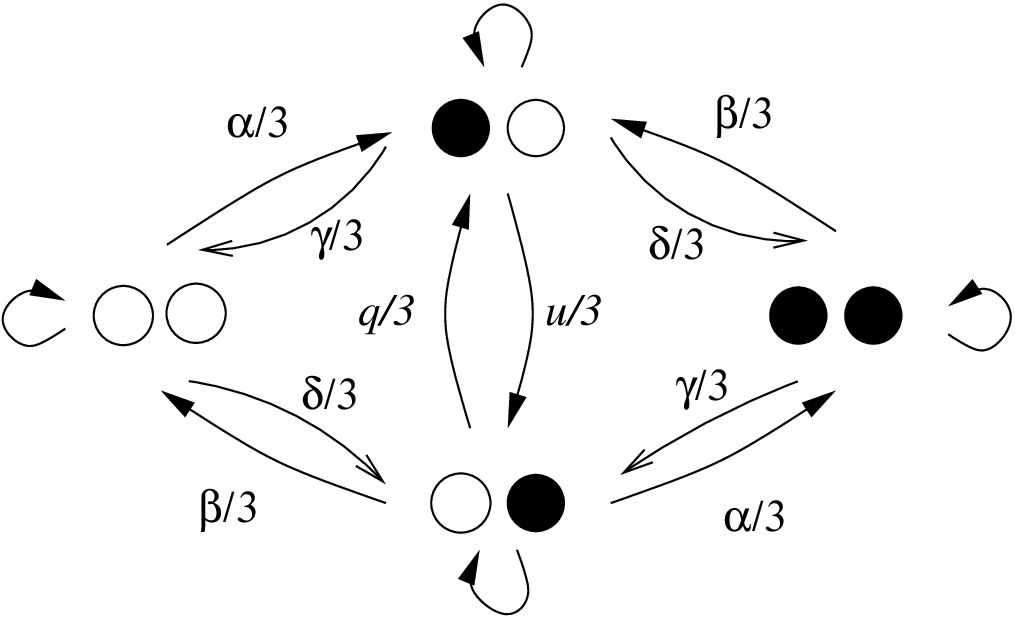}
\caption{The state diagram of the ASEP for $n=2$}
\label{states}
\end{figure}

In the long time limit, the system reaches a steady state where all 
the probabilities $P_n(\ttt_1, \ttt_2, \dots , \ttt_n)$ of finding
the system in configurations $(\ttt_1, \ttt_2, \dots , \ttt_n)$ are
stationary.  More specifically, the {\it stationary distribution}
is the unique (up to scaling) eigenvector of the transition
matrix of the Markov chain with eigenvalue $1$.

The ASEP clearly has multiple symmetries,
including the following.
\begin{itemize}
\item The ``left-right" symmetry: if we reflect the ASEP over the 
  $y$-axis,  
  we get back the same model,
  except that the parameters 
  $\alpha$ and $\delta$,
  $\gamma$ and $\beta$, and $u$ and $q$ are switched.
\item The ``arrow-reversal" symmetry: if we exchange black 
  and white particles, 
  we get back the same model, except that 
  the parameters $\alpha$ and $\gamma$,
  $\beta$ and $\delta$, and $u$ and $q$ are switched.
\item The ``particle-hole" symmetry: if we compose the above two
  symmetries, i.e. reflect the ASEP over the $y$-axis and 
  exchange black and white particles, we get back the same model,
  except that  $\alpha$ and $\beta$, and $\gamma$
  and $\delta$ are switched.
\end{itemize}

These symmetries imply results about the stationary distribution.
\begin{observation}\label{symmetries}
The steady state probabilities satisfy the following identities: 
\begin{itemize}
\item $P_n(\tau_1,\dots,\tau_n) = P_n(\tau_n,\dots,\tau_1) 
          \vert_{\alpha \leftrightarrow \delta, \beta \leftrightarrow \gamma, 
                  u \leftrightarrow q}$,
\item $P_n(\tau_1,\dots,\tau_n) = P_n(1-\tau_1,\dots,1-\tau_n) 
          \vert_{\alpha \leftrightarrow \gamma, \beta \leftrightarrow \delta, 
                  u \leftrightarrow q}$,
\item $P_n(\tau_1,\dots,\tau_n) = P_n(1-\tau_n,\dots,1-\tau_1) 
          \vert_{\alpha \leftrightarrow \beta, \gamma \leftrightarrow \delta}$.
\end{itemize}
\end{observation}
Above, the notation $\vert_{\alpha\leftrightarrow \delta}$ indicates
that the parameters $\alpha$ and $\delta$ are exchanged.
These symmetries are related to the symmetries of the Askey-Wilson
polynomials (see for example Remark \ref{AWSymmetry}), 
though neither is a direct consequence of the other.

\section{Staircase tableaux and the stationary distribution of the ASEP}\label{Results}

The main combinatorial objects of this paper are some new tableaux which we call 
{\it staircase tableaux}.  These tableaux generalize
permutation tableaux (equivalently, alternative tableaux).

\begin{definition}
A \emph{staircase tableau} of size $n$ is a Young diagram of ``staircase"
shape $(n, n-1, \dots, 2, 1)$ such that boxes are either empty or 
labeled with $\alpha, \beta, \gamma$, or $\delta$, subject to the following conditions:
\begin{itemize}
\item no box along the diagonal is empty;
\item all boxes in the same row and to the left of a $\beta$ or a $\delta$ are empty;
\item all boxes in the same column and above an $\alpha$ or a $\gamma$ are empty.
\end{itemize}
The \emph{type} of a staircase tableau is a word in $\{\bullet, \circ\}^n$ 
obtained by reading the diagonal boxes from 
northeast to southwest and writing a $\bullet$ for each $\alpha$ or $\delta$,
and a $\circ$ for each $\beta$ or $\gamma$.
\end{definition}

\begin{remark}\label{rem:type}
For convenience, we sometimes refer to the type of a staircase
tableau as a word in $\{D,E\}^n$ rather than $\{\bullet, \circ\}^n$, 
by identifying
$\circ$ by $E$ and $\bullet$ by $D$.
\end{remark}

See the left of Figure \ref{staircase} for an example of a staircase tableau.  

\begin{figure}[h]
\input{Staircase.pstex_t} \hspace{6em} \input{Staircase2.pstex_t}
\caption{A staircase tableau
of size $7$ and type $\circ \circ \bullet \bullet \bullet \circ \circ$}
\label{staircase}
\end{figure}
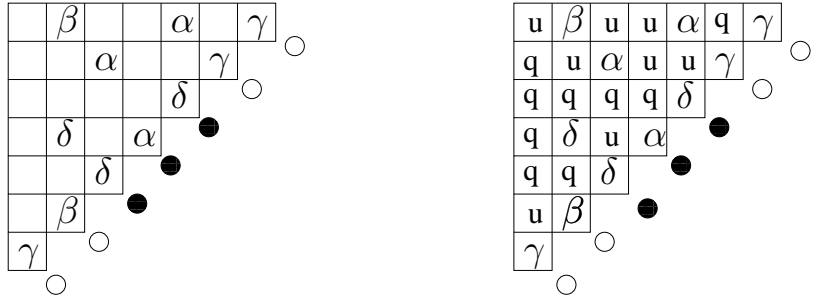

\begin{definition}\label{weight}
The \emph{weight} $\wt(\T)$ of a staircase tableau $\T$ is a monomial in 
$\alpha, \beta, \gamma, \delta, q$, and $u$, which we obtain as follows.
Every blank box of $\T$ is assigned a $q$ or $u$, based on the label of the closest
labeled box to its right in the same row and the label of the closest labeled box
below it in the same column, such that:
\begin{itemize}
\item every blank box which sees a $\beta$ to its right gets assigned a $u$;
\item every blank box which sees a $\delta$ to its right gets assigned a $q$;
\item every blank box which sees an $\alpha$ or $\gamma$ to its right,
  and an $\alpha$ or $\delta$ below it, gets assigned a $u$; 
\item every blank box which sees an $\alpha$ or $\gamma$ to its right,
  and a $\beta$ or $\gamma$ below it, gets assigned a $q$.
\end{itemize}
After assigning a $q$ or $u$ to each blank box in this way, the 
\emph{weight} of $\T$ is then defined as the product of all labels in all boxes.
\end{definition}

The right of Figure \ref{staircase} 
shows that this staircase tableau 
has weight $\alpha^3 \beta^2 \gamma^3 \delta^3 q^9 u^8$.

\begin{remark}
The weight of a staircase tableau  
always has degree $n(n+1)/2$.  For convenience, we will sometimes set $u=1$, 
since this results in no loss of information.
\end{remark}

Our first main result (to be proved in Section \ref{proof1}) is the following.

\begin{theorem}\label{NewThm}
Consider any state $\tau$ of 
the ASEP with $n$ sites, where the parameters 
$\alpha, \beta, \gamma, \delta, q$ and $u$ are general.
Set $Z_n = \sum_{\T} \wt(\T)$,
where the sum is over all staircase tableaux of size $n$.
Then 
the steady state probability that the 
ASEP is at state $\tau$ is 
precisely 
\begin{equation*}
\frac{\sum_{\T} \wt(\T)}{Z_n},
\end{equation*}
where the sum is over all staircase tableaux $\T$ of type $\tau$.
In particular, $Z_n$ is the partition function for the ASEP.
\end{theorem}

Figure \ref{tableaux} illustrates Theorem \ref{NewThm}
for the state $\bullet \bullet$
of the ASEP.  All staircase tableaux $\T$
of type $\bullet \bullet$ are shown.
It follows that the steady state probability of  $\bullet \bullet$ is 
\begin{equation*}
\frac{\alpha^2 u + \delta^2 q +\alpha \delta q  +\alpha \delta u +
\alpha^2 \delta+ \alpha \beta \delta +\alpha \gamma \delta+\alpha \delta^2}
{Z_2}.
\end{equation*}

\begin{figure}[h]
\input{Example.pstex_t}
\caption{The tableaux of type $\bullet \bullet$}
\label{tableaux}
\end{figure}

We can also obtain combinatorial formulas for various physical quantities.
Theorem \ref{Physical} below will be proved in Section \ref{Applications}.

\begin{theorem}\label{Physical}
Consider the ASEP with $n$ sites.  Then we have
the following:

\begin{itemize}
\item The current in the steady state  is 
$\frac{Z_{n-1}(\alpha \beta - \gamma \delta q^{n-1})}{Z_n},$ where $Z_n$ is the generating function
for staircase tableaux of size $n$.

\item The average particle number $\langle \tau_i \rangle_n$ 
at site $i$ is given by 
$Z_n^{-1}$ times the generating function for all staircase
tableaux of size $n$ which have an $\alpha$ or $\delta$
at the $i$th position along the diagonal.

\item Similarly, 
the $m$-point function $\langle \tau_{i_1} \dots \tau_{i_m} \rangle_n$
is given by $Z_n^{-1}$ times the generating function for all staircase
tableaux of size $n$ which have an $\alpha$ or $\delta$
in positions $i_1,i_2,\dots,$ and $i_m$ along the diagonal.
\end{itemize}
\end{theorem}

\begin{remark}
In \cite{CW1} and \cite{CW2}, which concerned the ASEP
with parameters $\gamma=\delta=0$, we gave combinatorial
expressions for the stationary distribution in terms of 
{\it permutation tableaux}: more specifically, 
the steady state probability of a state $\tau$ was given by 
enumerating permutation tableaux lying in a Young diagram of shape
$\lambda(\tau)$,
according to 
the number of $1$'s in the top row and the number of unrestricted 
rows. 

Subsequently \cite{Burstein} and 
\cite{CN} observed that a permutation
tableau is determined by complementary statistics,
namely the positions of the topmost $1$'s not in the first row,
and rightmost restricted
zero's.  Independently, Viennot \cite{Viennot}
defined some new \emph{alternative tableaux}, in order to give
a simpler formula for the stationary distribution of the ASEP
when $\gamma=\delta=0$ than the one in \cite{CW1, CW2} ; it turns
out that alternative tableaux can be obtained from permutation tableaux,
by making boxes red, blue, or empty,
based on whether the box of the associated permutation tableau
contains a topmost $1$, rightmost restricted zero, or neither.
Our staircase tableaux generalize both kinds of tableaux, but now 
have non-empty boxes labeled by  
$\alpha, \beta, \gamma, \delta$.
Additionally, instead of working with Young diagrams of various shapes,
we now always work with a staircase shape, whose \emph{type}
encodes the same information (namely, the corresponding state of the ASEP)
that the shape used to.
Working
with this larger shape seems to be
the only natural way to assign the appropriate powers of 
$q$ and $u$ to the tableau.
\end{remark}

\begin{remark}
As an alternative to Definition \ref{weight}, suppose we
define the \emph{dual weight} $\wt'(\T)$ of a staircase
tableau by taking the product of all labels in all boxes,
after having filled the blank boxes of $\T$ according to the 
following rule:
\begin{itemize}
\item every blank box which sees an $\alpha$ below it gets assigned a $u$;
\item every blank box which sees a $\gamma$ below it gets assigned a $q$;
\item every blank box which sees an $\alpha$ or $\delta$ to its right,
  and a $\beta$ or $\delta$ below it, gets assigned a $q$; 
\item every blank box which sees a $\beta$ or $\gamma$ to its right,
  and a $\beta$ or $\delta$ below it, gets assigned a $u$.
\end{itemize}
Then Theorem \ref{NewThm} continues to hold, with $\wt$ replaced
by $\wt'$.  This follows from the left-right
symmetry of the ASEP.
More specifically, note that if $\T$ is a staircase tableau, then 
the tableau
$\T'$ obtained by  transposing $\T$ then switching
$\alpha$ and $\delta$, and $\beta$ and $\gamma$ is still a staircase tableau.  
Our observation now follows
from Theorem \ref{NewThm}, the fact that $\wt'(\T') = \wt(\T)$,
and Observation \ref{symmetries}.
(Alternatively, we could have proved the observation using 
a method analogous to the proof of Theorem 
\ref{NewThm}.)
\end{remark}

\section{Askey-Wilson polynomials and a formula for their moments}\label{AW Section}

The Askey-Wilson polynomials are  orthogonal polynomials
with five free parameters ($a, b, c, d, q$).
They reside at the top of the hierarchy of the one-variable
orthogonal polynomial family in the Askey scheme \cite{AW,GR,Koekoek}.
In this section we define the Askey-Wilson polynomials, following
the exposition of \cite{AW} and \cite{USW}, then  
state a combinatorial formula for their moments.

The $q$-shifted factorial is defined by
\begin{eqnarray*}
(a_1,a_2,\cdots,a_s;q)_n=\prod_{r=1}^s \prod_{k=0}^{n-1} (1-a_rq^k),
\end{eqnarray*}
and the basic hypergeometric function is given by 
\begin{eqnarray*}
{}_r\phi_s\left[ {{a_1,\cdots ,a_r}\atop{b_1,\cdots ,b_s}};q,z\right]
=\sum_{k=0}^\infty \frac{(a_1,\cdots,a_r;q)_k}{(b_1,\cdots,b_s,q;q)_k}
((-1)^k q^{k(k-1)/2})^{1+s-r} z^k.
\end{eqnarray*}

The Askey-Wilson polynomial $P_n(x)=P_n(x;a,b,c,d\vert q)$
is explicitly defined by
\begin{eqnarray*}
P_n(x)=
a^{-n}(ab,ac,ad;q)_n\
{}_4\phi_3\left[ {{q^{-n},q^{n-1}abcd,ae^{i\theta},ae^{-i\theta}}
        \atop{ab,ac,ad}};q,q \right] ,
\label{eqn:defAW}
\end{eqnarray*}
with $x=\cos\theta$ for $n\in \Z_+:=\{0,1,2,\cdots\}$.
It satisfies the three-term recurrence
\begin{eqnarray*}
A_nP_{n+1}(x)+B_nP_n(x)+C_nP_{n-1}(x)=2xP_n(x),
\label{eqn:recAW}
\end{eqnarray*}
with $P_0(x)=1$ and $P_{-1}(x)=0$,
where
\begin{align*}
A_n&=
\frac{1-q^{n-1}abcd}{(1-q^{2n-1}abcd)(1-q^{2n}abcd)},
\\
B_n&=
\frac{q^{n-1}}{(1-q^{2n-2}abcd)(1-q^{2n}abcd)}
[(1+q^{2n-1}abcd)(qs+abcds')-q^{n-1}(1+q)abcd(s+qs')],
\\
C_n&=
\frac{(1-q^n)(1-q^{n-1}ab)(1-q^{n-1}ac)(1-q^{n-1}ad)(1-q^{n-1}bc)
(1-q^{n-1}bd)(1-q^{n-1}cd)}
{(1-q^{2n-2}abcd)(1-q^{2n-1}abcd)},
\end{align*}
\begin{eqnarray*}
\text{ and }~~s=a+b+c+d, \qquad s'=a^{-1}+b^{-1}+c^{-1}+d^{-1}.
\end{eqnarray*}

\begin{remark}\label{AWSymmetry}
It is obvious from the three-term recurrence that the polynomials
$P_n(x)$ are symmetric in $a, b, c$ and $d$.
\end{remark}

For $|a|, |b|, |c|, |d| < 1$, using 
$z=e^{i\theta}$,
the orthogonality is expressed by 
\begin{eqnarray*}
\oint_C \frac{dz}{4\pi iz} w\left(\frac{z+z^{-1}}{2}\right)
P_m\left(\frac{z+z^{-1}}{2}\right)P_n\left(\frac{z+z^{-1}}{2}\right)
=h_n\delta_{mn},
\label{eqn:orthoointAW}
\end{eqnarray*}
where the integral contour $C$ is a closed path which
encloses the poles at $z=aq^k$, $bq^k$, $cq^k$, $dq^k$ $(k\in \Z_+)$
and excludes the poles at $z=(aq^k)^{-1}$, $(bq^k)^{-1}$, $(cq^k)^{-1}$,
$(dq^k)^{-1}$ $(k\in \Z_+)$, and where
\begin{eqnarray*}
&&w(\cos\theta)=
\frac{(e^{2i\theta},e^{-2i\theta};q)_\infty}
{(ae^{i\theta},ae^{-i\theta},be^{i\theta},be^{-i\theta},
ce^{i\theta},ce^{-i\theta},de^{i\theta},de^{-i\theta};q)_\infty}, \\
&&\frac{h_n}{h_0}=
\frac{(1-q^{n-1}abcd)(q,ab,ac,ad,bc,bd,cd;q)_n}
{(1-q^{2n-1}abcd)(abcd;q)_n} ,\\
&&h_0=
\frac{(abcd;q)_\infty}{(q,ab,ac,ad,bc,bd,cd;q)_\infty} .
\end{eqnarray*}
(In the other parameter region, the orthogonality is continued analytically.)

The moments are defined by 
\begin{eqnarray*}
\mu_k = \oint_C \frac{dz}{4\pi iz} w\left(\frac{z+z^{-1}}{2}\right)
\left(\frac{z+z^{-1}}{2}\right)^k.
\label{eqn:moment}
\end{eqnarray*}

The second main result of this paper is a combinatorial formula for 
the moments of the Askey-Wilson polynomials.
In Theorem \ref{moments} below, we use the substitution 
\begin{align*}
\alpha&=\frac{1-q}{1+ac+a+c},~~~~~ 
&\beta=\frac{1-q}{1+bd+b+d},\\
\gamma&=\frac{-(1-q)ac}{1+ac+a+c},~~~~~
&\delta=\frac{-(1-q)bd}{1+bd+b+d},
\end{align*}
which can be inverted via 
\begin{eqnarray*}
a&=&\frac{1-q-\alpha+\gamma+\sqrt{(1-q-\alpha+\gamma)^2+4\alpha\gamma}}{2\alpha}\\
c&=&\frac{1-q-\alpha+\gamma-\sqrt{(1-q-\alpha+\gamma)^2+4\alpha\gamma}}{2\alpha}\\
b&=&\frac{1-q-\beta+\delta+\sqrt{(1-q-\beta+\delta)^2+4\beta\delta}}{2\beta}\\
d&=&\frac{1-q-\beta+\delta-\sqrt{(1-q-\beta+\delta)^2+4\beta\delta}}{2\beta}.\\
\end{eqnarray*}

Recall that $Z_{\ell}=\sum_{\T} \wt(\T)$, where the sum is over all 
staircase tableaux of size $\ell$.
\begin{theorem}\label{moments}
The $k$th moment of the Askey-Wilson polynomials is given by 
\begin{equation*}
\mu_k = h_0 \sum_{\ell=0}^k (-1)^{k-\ell}{k \choose \ell} 
\left(\frac{1-q}{2}\right)^{\ell}
 \frac{{Z}_{\ell}}{\prod_{i=0}^{\ell-1} (\alpha \beta - \gamma \delta q^i)}.
\end{equation*}
\end{theorem}

\section{A more flexible Matrix Ansatz}\label{newansatz}

One of the most powerful techniques for studying
the ASEP is the so-called {\it Matrix Ansatz},
an Ansatz given by Derrida, Evans, Hakim, and Pasquier \cite{Derrida1} 
as a tool for solving for the 
steady state probabilities 
$P_n(\ttt_1, \dots , \ttt_n)$
of the ASEP.  In this section we 
will start by recalling their Matrix Ansatz,
and then give a new generalization of it which we require for our proof of 
Theorem \ref{NewThm}.

For convenience, in this section we set $u=1$.  Also, we
define unnormalized weights $f_n(\ttt_1, \dots , \ttt_n)$, which are
equal to the $P_n(\ttt_1, \dots , \ttt_n)$ up to a constant:
\begin{equation*}
P_n(\ttt_1, \dots , \ttt_n) = f_n(\ttt_1, \dots , \ttt_n)/Z_n,
\end{equation*}
where
$Z_n$ is the {\it partition function}
$\sum_{\tau} f_n(\ttt_1, \dots , \ttt_n)$.  The sum defining
$Z_n$ is
over all possible configurations $\ttt \in \{0,1\}^n$.
Derrida {\it et al} showed the following.

\begin{theorem} \cite{Derrida1} \label{ansatz}
Suppose that $D$ and $E$ are matrices,  $V$ is a column vector,
and $W$ is a row vector, with $WV=1$,
such that the following conditions hold:
\begin{equation*}
 DE - qED = D+E,~~~~~~~
 \beta DV-\delta EV =  V,~~~~~~~
 \alpha WE-\gamma WD =  W.
\end{equation*}
Then for any state $\tau=(\tau_1,\dots,\tau_n)$ of the ASEP,
\begin{equation*}
f_n(\ttt_1, \dots , \ttt_n) = W (\prod_{i=1}^n (\ttt_i D + (1-\ttt_i)E))V.
\end{equation*}
\end{theorem}

Note that $\prod_{i=1}^n (\ttt_i D + (1-\ttt_i)E)$ is simply
a product of $n$ matrices $D$ or $E$ with matrix $D$ at position $i$
if site $i$ is occupied ($\ttt_i=1)$.
Also note that Theorem \ref{ansatz} implies that 
$Z_n=W (D+E)^n V$.

We now state and prove a more flexible version of 
Theorem \ref{ansatz}.  Our proof generalizes the argument
given in \cite{Derrida1}.

\begin{theorem} \label{ansatz2}
Let $\{\lambda_n\}_{n \geq 0}$ be a family of constants.  Let $W$
and $V$ be row and column vectors, with $WV=1$,
and let $D$ and $E$ be matrices
such that for any words $X$ and $Y$ in $D$ and $E$, we have:
\begin{enumerate}
\item[(I)] $WX(DE-qED)YV = \lambda_{|X|+|Y|+2} WX(D+E)YV$;
\item[(II)] $\beta WXDV-\delta WXEV=\lambda_{|X|+1} WXV$;
\item[(III)] $\alpha WEYV-\gamma WDYV=\lambda_{|Y|+1} WYV$.
\end{enumerate}
(Here $|X|$ is the length of $X$.)  Then for any state
$\tau=(\tau_1,\dots,\tau_n)$ of the ASEP, 
\begin{equation}\label{probs}
f_n(\tau)={W (\prod_{i=1}^n (\tau_i D+(1-\tau_i)E)) V}.
\end{equation}
\end{theorem}

\begin{proof}
We are in the steady state of the ASEP if the net rate of entering 
each state $(\tau_1,\dots,\tau_n)$ is 0, or in other words, 
the following expression equals $0$:
\begin{eqnarray}
&&(-1)^{\tau_1} (-\alpha f_n(0,\tau_2,\dots,\tau_n) +  
\gamma f_n(1,\tau_2,\dots,\tau_n)) \label{First}  \\
&& +\sum_{i=1}^{n-1} (-1)^{\tau_i} \chi(\tau_i \neq \tau_{i+1})
  (f_n(\tau_1,\dots,1,0,\dots,\tau_n)-qf_n(\tau_1,\dots,0,1,\dots,\tau_n)) \label{second}\\
&&+(-1)^{\tau_n} (\beta f_n(\tau_1,\dots,\tau_{n-1},1)-\delta f_n(\tau_1,\dots,
\tau_{n-1},0)). \label{third}
\end{eqnarray}
In (\ref{second}) above, the arguments $1,0$ and $0,1$ are in 
positions $i$ and $i+1$, and $\chi$ is the boolean function taking value $1$ or $0$
based on whether its argument is true or false.
So what we have to prove is that the quantities in the right-hand-side
of equation (\ref{probs}) satisfy this equation.

By the assumptions (I), (II), and (III) of the theorem, we have the following:
\begin{itemize}
\item the expression (\ref{First})  equals 
   $\pm \lambda_n f_{n-1}(\tau_2,\dots,\tau_n)$, based on whether
   $\tau_1$ is $1$ or $0$;
\item (\ref{second}) equals $0$ when $\tau_i=\tau_{i+1}$; otherwise, based
  on whether $\tau_i$ is $1$ or $0$,  
  it equals 
  $\mp \lambda_n \sum_{i=1}^n (f_{n-1}(\tau_1,\dots,\hat{\tau}_i,\dots,\tau_n)
  + f_{n-1}(\tau_1,\dots,\hat{\tau}_{i+1},\dots,\tau_n))$;
\item and  (\ref{third})  equals
 $ \mp \lambda_n f_{n-1}(\tau_1,\dots,\tau_{n-1})$, based on whether 
 $\tau_n$ is $1$ or $0$.
\end{itemize}
(Here $\hat{\tau_i}$ denotes the omission of the $i$th component.)
Then using these conditions, it is easy to verify that the sum of 
(\ref{First}), (\ref{second}), and (\ref{third}) is equal to $0$,
since all terms involving $f_{n-1}$ cancel out.

\end{proof}

\section{The proof of the stationary distribution}\label{proof1}

In this section we will prove Theorem \ref{NewThm}
by: defining vectors $W,V$ and matrices 
$D,E$; proving
that they have the requisite combinatorial
interpretation in terms of staircase tableaux;
and checking that they satisfy the relations of Theorem
\ref{ansatz2}, with
$\lambda_0=1$ and $\lambda_n=\alpha \beta-\gamma \delta q^{n-1}$ for $n\geq 1$.

This is analogous to the proof of \cite[Theorem 3.1]{CW1}, 
albeit much more difficult: in \cite{CW1}, it was obvious 
that our matrices and vectors satisfied the Matrix Ansatz
\cite[Lemma 2.5]{CW1}, and easy to show
that our combinatorial 
objects were described by the algebraic relations of the Ansatz,
see \cite[Figure 6]{CW1} and the 
surrounding discussion.

In contrast, in this more general situation, 
we can give a combinatorial proof of relation (III)
of our new Matrix Ansatz, but not
for (I) or (II).  Instead we give a rather difficult algebraic proof of 
(I) and (II).
First of all,
our new ``vectors" and ``matrices"
have two and four indices, respectively, which makes
working with them more complicated.  Second,
to use Theorem \ref{ansatz2}, instead of proving
that our vectors and matrices satisfy three identities
(as in Theorem \ref{ansatz}), we must prove
that they satisfy three {\it infinite families} of identities.
Moreover, there is no obvious way to use induction to prove 
these identities: one cannot take one of the identities and 
multiply on the left or right to obtain the next identity 
in the family.

\begin{remark}\label{u=1}
In this section we assume $u=1$.  Recall that this is no
loss of generality, as the weight of a staircase tableau
of size $n$ is always a monomial of degree $n(n+1)/2$.
\end{remark}

\subsection{The definition of our matrices}\label{def-matrices}

\begin{definition}  In what follows, indices range over the non-negative
integers.  In particular, our matrices and vectors are not finite.
We define ``row" and ``column" vectors $W=(W_{ik})_{i,k}$ and 
$V=(V_{j\ell})_{j,\ell}$, and matrices
$D=(D_{i,j,k,\ell})_{i,j,k,\ell}$
and $E=(E_{i,j,k,\ell})_{i,j,k,\ell}$  
by the following:
\begin{eqnarray*}
W_{ik}&=&\left\{
\begin{array}{ll}
1 &\textup{if $i=k=0$,}\\
0  &\textup{otherwise,}
\end{array}  \right.\\
V_{j\ell}&=&1 \text{ always.}
\end{eqnarray*}
\begin{eqnarray*}
D_{i,j,k,\ell}&=&\left\{
\begin{array}{ll}
0 &\textup{if $j<i$ or $\ell>k+1$,}\\
\delta q^i &\textup{if $i=j-1$ and $k=\ell=0$,}\\
\alpha q^i  &\textup{if $i=j$, $k=0$ and $\ell=1$,}\\
\delta (D_{i,j-1,k-1,\ell}+E_{i,j-1,k-1,\ell})+D_{i,j,k-1,\ell-1}  
&\textup{otherwise.}
\end{array}  \right.\\
E_{i,j,k,\ell}&=&\left\{
\begin{array}{ll}
0 &\textup{if $j<i$ or $\ell>k+1$,}\\
\beta q^i &\textup{if $i=j$ and $k=\ell=0$,}\\
\gamma q^i  &\textup{if $i=j$, $k=0$ and $\ell=1$,}\\
\beta (D_{i,j,k-1,\ell}+E_{i,j,k-1,\ell})+qE_{i,j,k-1,\ell-1}\hspace{.2cm}~~~~~~~~~~  
&\textup{otherwise.}
\end{array}  \right.
\end{eqnarray*}

\end{definition} 

By convention, if any subscript $i, j, k$ or $\ell$ is negative, then
$D_{ijk\ell}=E_{ijk\ell} = 0$.

\begin{example}
\begin{align*}
E_{0,2,2,0} &= \beta (D_{0,2,1,0}+E_{0,2,1,0})+qE_{0,2,1,-1} \\ 
            &= \beta (D_{0,2,1,0}+E_{0,2,1,0})\\ 
            &= \beta[ \delta(D_{0,1,0,0}+E_{0,1,0,0})+D_{0,2,0,-1}+\beta(D_{0,2,0,0}+E_{0,2,0,0})+qE_{0,2,0,-1}]\\
            &= \beta \delta(D_{0,1,0,0}+E_{0,1,0,0})+\beta^2(D_{0,2,0,0}+E_{0,2,0,0})\\
            &= \beta \delta(\delta + 0) +\beta^2(0+0) = \beta \delta^2
\end{align*}
\end{example}

Here, we think of the two coordinates
$i$ and $k$ as specifying a ``row" of a matrix, and the two 
coordinates $j$ and $\ell$ as specifying a ``column" of a matrix. 
Therefore matrix multiplication is defined by 
$$(MN)_{i,j,k,\ell} = \sum_{a,b} M_{i,a,k,b} N_{a,j,b,\ell}.$$
Note that 
when $M$ and $N$ are products of $D$'s and $E$'s the sum on the right-hand-side
is  finite.  
Specifically, if $M$ is a word of length $r$ in $D$ and $E$,
then $M_{i,a,k,b}$ is $0$ unless $a+b \leq i + k+r$.  This can
be shown by induction from the definition of $D$ and $E$, or 
by using the combinatorial interpretation 
given in Lemma \ref{comb-lemma}.

\subsection{The combinatorial interpretation of our matrices in terms
of tableaux}

We say that a row of a staircase tableau $\T$
is {\it indexed by $\beta$} if the leftmost box 
in that row which is not occupied by a $q$ or $u$
is a $\beta$.  Note that every box to the left of that $\beta$
must be a $u$.  Similarly we will talk about rows which 
are {\it indexed by $\delta$}; in this case, every box to the left of that 
$\delta$ must be a $q$.  We will also talk about rows
which are {\it indexed by $\alpha/\gamma$}, which is shorthand
for rows which are indexed by $\alpha$ {\it or} $\gamma$.


\begin{theorem}\label{comb-interpret}
If $X$ is a word in $D$'s and $E$'s, then:
\begin{itemize} 
\item $X_{ijk\ell}$ is the generating function for 
all ways of adding $|X|$ new columns of type $X$ to a staircase
tableau with $i$ rows indexed by $\delta$ and $k$ rows indexed by $\alpha/\gamma$,
so as to obtain a new tableau with $j$ rows indexed by 
$\delta$ and $\ell$ rows indexed by $\alpha/\gamma$.
\item $(WX)_{j\ell}$ is the generating function for 
staircase tableaux of type $X$ which have $j$ rows
indexed by $\delta$ and $\ell$ rows indexed by $\alpha/\gamma$
(and hence $|X|-j-\ell$ rows indexed by $\beta$.)
\item $WXV$ is the generating function for all staircase
tableaux of type $X$.
\end{itemize}
\end{theorem}

The main step in proving Theorem \ref{comb-interpret}
is the
following lemma, which
says that the matrices $D$ and $E$
are ``transfer matrices" for building
staircase tableaux.  

\begin{lemma}\label{comb-lemma}
$D_{i,j,k,\ell}$ is the generating function for the weights
of all possible new columns with an $\alpha$ or $\delta$ in the bottom box
that we could add to the left of a staircase tableau with
$i$ rows indexed by $\delta$ and $k$ rows indexed by $\alpha$ or $\gamma$,
obtaining a new staircase tableau which has $j$ rows indexed by 
$\delta$ and $\ell$ rows indexed by $\alpha$ or $\gamma$.
Similarly for $E_{i,j,k,\ell}$, where the new column has a $\beta$
or $\gamma$ in the bottom box. 
\end{lemma}

\begin{proof}
Let $D'_{ijk\ell}$ denote the generating function
for all possible new columns with an $\alpha$ or $\delta$ in the bottom box
that we could add to the left of a staircase tableau with
$i$ rows indexed by $\delta$ and $k$ rows indexed by $\alpha/\gamma$,
obtaining a new staircase tableau which has $j$ rows indexed by
$\delta$ and $\ell$ rows indexed by $\alpha$ or $\gamma$.
We will show that $D'_{ijk\ell} = D_{ijk\ell}$ by showing
that $D'$ satisfies the same recurrences.

Note that $D'_{ijk\ell}=0$ if $j<i$ because
adding a new column to a staircase tableau never 
decreases the number of rows indexed by $\delta$.
Also $D'_{ijk\ell}=0$ if $\ell > k+1$ because when
we add a new column we can never increase the number
of rows indexed by $\alpha/\gamma$ by more than $1$.

Now suppose that $k=0$.  If we are starting from
a tableau with $i$ rows indexed by $\delta$
and $0$ rows indexed by $\alpha/\gamma$, 
then the only way to add
a new column is to add a column with an $\alpha$ or $\delta$ 
at the bottom, with all boxes above empty.
If we add a $\delta$, then the resulting tableau
has $\ell=0$ rows indexed by $\alpha/\gamma$ and $j=i+1$
rows indexed by $\delta$.  The weight of the new column
will be $\delta q^i$.
On the other hand, if we add an $\alpha$ at the bottom,
then the resulting tableau has $\ell=1$ rows indexed by
$\alpha/\gamma$ and $j=i$ rows indexed by $\delta$.  The weight of 
the new column will be $\alpha q^i$.
From this discussion it follows that 
$D'_{ijk\ell}=\delta q^i$ when $j=i+1$ and $k=\ell=0$,
and $D'_{ijk\ell}=\alpha q^i$ when $j=i$, $k=0$, and $\ell=1$.

In all other situations, we can assume that $k \geq 1$.
Suppose that we are adding a new column $C$
with an $\alpha$ or $\delta$ at the bottom to the left of a staircase
tableau with $i$ rows indexed by $\delta$ and $k$
rows indexed by $\alpha/\gamma$, so as to create a new tableau $\T$.
Consider the lowest box $B$ of $C$ whose row in $\T$ is indexed by an
$\alpha$ or $\gamma$ (such a box exists since $k \geq 1$).
If we fill $B$ with an $\alpha, \beta, \gamma$ or $\delta$,
then the bottom box of $C$ must contain a $\delta$.
In this case, if we ignore that bottom $\delta$,
then our choices for $C$ are exactly 
the same as our choices would be for adding a 
new column to the left of a staircase tableau
with $i$ rows indexed by $\delta$ and $k-1$
rows indexed by $\alpha/\gamma$.
Therefore, filling $B$ with an $\alpha, \beta, \gamma$ or $\delta$
gives us a contribution of 
$d(D'+E')_{i,j-1,k-1,\ell}$ to our generating function.

On the other hand, if we leave $B$ empty,
then this box will get a weight $u=1$ (recall Remark \ref{u=1}).
Filling the rest of the column $C$ is like
adding a new column to a staircase tableau
with $i$ rows indexed by $\delta$ and $k-1$ rows indexed
by $\alpha/\gamma$.  Therefore leaving $B$ empty 
gives us a contribution of 
$D'_{i,j,k-1,\ell-1}$ to our generating function.

It follows that when $k \geq 1$,
$D'_{ijk\ell} = \delta (D'+E')_{i,j-1,k-1,\ell} + D'_{i,j,k-1,\ell-1}.$

Similarly, 
we define $E'_{ijk\ell}$ to be the generating function for all possible
new columns with a $\beta$ or $\gamma$ in the bottom box that we could
add to the left of a staircase tableau with $i$ rows indexed by $\delta$
and $k$ rows indexed by $\alpha/\gamma$, obtaining a new staircase tableau
which has $j$ rows indexed by $\delta$ and $\ell$ rows indexed by 
$\alpha/\gamma$.  
The proof that $E'_{ijk\ell}=E_{ijk\ell}$ is 
analogous to the proof we gave for $D'$.
\end{proof}

\begin{proof}[Proof of Theorem \ref{comb-interpret}]
The first item follows from Lemma \ref{comb-lemma} and the definition
of matrix multiplication.
Multiplying at the left by a $W$ has the effect that we start 
with the empty tableau and then add columns according to $X$:
so 
$W X_{j\ell}$ is the generating function
for staircase tableaux of type $X$ which have
$j$ rows indexed by $\delta$ and $\ell$ rows indexed by 
$\alpha/\gamma$.  Finally, multiplying $WX_{j\ell}$ on the right by $V$
has the effect of summing over all $\delta$ and $\ell$,
so $WXV$ is the generating function for all staircase
tableaux of type $X$.
\end{proof}

\subsection{The proof that our matrices satisfy the Matrix Ansatz}

We now prove that our matrices
satisfy  
Theorem \ref{ansatz2}, with $\lambda_n=\alpha \beta-\gamma \delta q^{n-1}$ for $n \geq 1$.
Relation (III) has a simple combinatorial
proof.  However, this proof does {\it not} work for relation (II),
and indeed it will require a lot more work to prove
 (I) and (II).

\begin{lemma}
Relation (III) of Theorem \ref{ansatz2} holds.
\end{lemma}
\begin{proof}
Using Theorem \ref{comb-interpret}, relation (III) can be
reformulated in terms of staircase tableaux.  
First we rewrite (III) as 
$$\alpha WE YV +  \gamma \delta q^{n-1} WYV = \gamma WD YV  + \alpha \beta WYV,$$
where $n-1=|Y|$.
Since a ``type E" corner box of a staircase tableau
must be either a $\beta$ or $\gamma$, we can rewrite this again as
$$\alpha WE^\beta YV + \alpha WE^\gamma  YV + \gamma \delta 
q^{n-1} WYV = \gamma WD^\alpha YV + \gamma WD^\delta  YV + \alpha \beta WYV.$$
Here $WE^\beta YV$ denotes the generating function for staircase tableaux
of type $EY$, whose northeast corner box is a $\beta$;  
the terms $WE^\gamma  YV$, $WD^\alpha YV$, and $WD^\delta YV$ are defined 
analogously.\footnote{We could 
have defined matrices $D^\alpha, D^\delta, E^\beta, E^\gamma$ so
that they have
this combinatorial interpretation, and then set
$D=D^\alpha + D^\delta$ and $E=E^\beta+E^\gamma$.}

It is now easy to see that 
$\alpha WE^\beta YV = \alpha \beta WYV$ and $\gamma WD^\alpha YV = 
\gamma \delta q^{n-1} WYV$,
since a box labeled $\beta$ must have only empty boxes
(weighted $u=1$) to its left, and a box labeled $\delta$
must have only empty boxes (weighted $q$) to its left.
Also, since the rules for the weight of an empty box which sees
a $\gamma$ to its right are the same as the rules for the weight of 
an empty box which sees an $\alpha$ to its right, we have that 
$\alpha WE^\gamma YV = \gamma W D^\alpha YV.$  This proves relation (III).
\end{proof}

\begin{lemma}\label{decrease}
For any word $Y$ in $D$ and $E$, we have 
$Y_{ijk\ell} = q^{|Y|} Y_{i-1,j-1,k,\ell}$.
\end{lemma}
\begin{proof}
We use Theorem \ref{comb-interpret}.  Note that 
both $Y_{ijk\ell}$ and $Y_{i-1,j-1,k,\ell}$ enumerate
the ways of adding $|Y|$ new columns to a staircase tableau
$\T$ so as to 
increase by $j-i$ the number of rows indexed by $\delta$, and to increase by 
$\ell-k$ the number of rows indexed by $\alpha/\gamma$.  The only difference
is the initial number of rows indexed by $\delta$.  Since 
$Y_{ijk\ell}$ has one extra initial row indexed by $\delta$,
this will contribute $|Y|$ extra empty boxes which all get
the weight $q$.  Therefore 
$Y_{ijk\ell} = q^{|Y|} Y_{i-1,j-1,k,\ell}$.
\end{proof}

\begin{proposition}\label{suffices}
To prove (I) and (II), it suffices to prove the following identities
for all non-negative integers $j$ and $\ell$:
\begin{enumerate}
\item $(WXDE)_{j\ell} = q(WXED)_{j\ell} + 
    \alpha \beta (WX(D+E))_{j\ell} - \gamma \delta q^{|X|+1} (WX(D+E))_{j-1,\ell}.$
\item $\beta (WXD)_{j\ell} = \delta (WXE)_{j-1,\ell} +\alpha \beta (WX)_{j,\ell-1}
  -\gamma \delta q^{|X|}(WX)_{j-1,\ell-1}.$
\end{enumerate}
\end{proposition}

\begin{proof}
We claim the following: 
if (1) is true, then for any word $Y$ in $D$'s and $E$'s, 
$(WXDEY)_{j\ell}$ is equal to 
$$q(WXEDY)_{j\ell} +
\alpha \beta (WX(D+E)Y)_{j\ell}-\gamma \delta q^{|X|+|Y|+1}
 (WX(D+E)Y)_{j-1,\ell}.$$
To prove the claim, let $Y$ be any word in $D$ and $E$.
Then 
$(WXDEY)_{j\ell}$ is equal to: 
{\small
\begin{eqnarray*}
&& \sum_{i,k} (WXDE)_{ik} Y_{ijk\ell}\\
&=&\sum_{i,k} q(WXED)_{ik} Y_{ijk\ell} + 
      \alpha \beta(WX(D+E))_{ik} Y_{ijk\ell}
- \gamma \delta q^{|X|+1}\sum_{ik}
      (WX(D+E))_{i-1,k} Y_{ijkl} \\
&=&q (WXEDY)_{j\ell} + \alpha \beta(WX(D+E)Y)_{j\ell} -
      \gamma \delta q^{|X|+|Y|+1} (WX(D+E)Y)_{j-1,\ell}.
\end{eqnarray*}
}
To deduce the final equality above, we applied Lemma \ref{decrease}
to the last term.

Now note that if we take the equation of the claim, and sum over all 
$j$ and $\ell$, then we get precisely (I) (since multiplication
on the right by $V$ has the effect of summing over all indices).
And if we take (2) and sum over all $j$ and $\ell$, we get (II)
This completes the proof.
\end{proof}

\begin{lemma}\label{suffices2}
If the identity (2) of Proposition \ref{suffices} holds for all $j$ and $\ell$,
then the identity (1) of Proposition \ref{suffices}  holds for all $j$ and $\ell$.
\end{lemma}
\begin{proof} To prove the lemma, note
that $(WXDE)_{j\ell} = 
\sum_{i,k} (WXD)_{ik} E_{ijk\ell}$ 
equals:
\begin{equation}
\frac{\delta}{\beta} \sum_{i,k} (WXE)_{i-1,k} E_{ijk\ell} +
     \alpha \sum_{i,k} (WX)_{i,k-1} E_{ijk\ell}
     -\frac{\gamma \delta q^{|X|}}{\beta}  \sum_{i,k} (WX)_{i-1,k-1} E_{ijk\ell} \label{10}  
\end{equation}
\begin{eqnarray}
&\hspace{.5cm}=& \frac{q\delta}{\beta} (WXEE)_{j-1,\ell} + \alpha\sum_{i,k}
          (WX)_{i,k-1} (\beta (D+E)_{i,j,k-1,\ell}
            +qE_{i,j,k-1,\ell-1})\label{11} \\
     &&-\frac{\gamma \delta q^{|X|}}{\beta} \sum_{i,k} 
       (WX)_{i-1,k-1} (q\beta (D+E)_{i-1,j-1,k-1,\ell} + 
          q^2 E_{i-1,j-1,k-1,\ell-1}) \nonumber \\
&\hspace{.5cm}=&q\beta^{-1} \delta (WXEE)_{j-1,\ell} + \alpha \beta(WX(D+E))_{j\ell}
 +\alpha q(WXE)_{j,\ell-1} \label{12}\\
&&- \gamma \delta q^{|X|+1}
   (WX(D+E))_{j-1,\ell}
  -\beta^{-1} \gamma \delta q^{|X|+2}(WXE)_{j-1,\ell-1} \nonumber \\
&\hspace{.5cm}=&(q(WXED)_{j\ell}-\alpha q(WXE)_{j,\ell-1}+q\beta^{-1}\gamma \delta q^{|X|+1}
 (WXE)_{j-1,\ell-1}) \label{13} \\
 &&+\alpha \beta(WX(D+E))_{j\ell}
 +\alpha q(WXE)_{j,\ell-1} -\gamma \delta q^{|X|+1}
   (WX(D+E))_{j-1,\ell}\nonumber \\
  &&-\beta^{-1}\gamma \delta q^{|X|+2}(WXE)_{j-1,\ell-1}\nonumber \\
&\hspace{.5cm}=&q(WXED)_{j\ell}+\alpha \beta(WX(D+E))_{j\ell}
  -\gamma \delta q^{|X|+1}(WX(D+E))_{j-1,\ell}. \nonumber
\end{eqnarray}

In the arguments above, to go from (\ref{10}) to (\ref{11}), we used 
Lemma \ref{decrease} to 
replace $E_{ijk\ell}$ in the first term
by $qE_{i-1,j-1,k,\ell}$.  
\end{proof}

By Proposition \ref{suffices} and Lemma \ref{suffices2}, 
to prove Theorem \ref{NewThm}, it 
is enough to prove the following.

\begin{theorem}\label{identities}
For every word $X$ in $D$ and $E$,
and all $j,  \ell \in \Z_{\geq 0}$, 
the identity 
(2) of Proposition \ref{suffices} holds.
Equivalently, 
\begin{equation*}
\beta(XD)_{0,j,0,\ell} - \delta(XE)_{0,j-1,0,\ell} - \alpha \beta (X)_{0,j,0,\ell-1} + q^{|X|} \gamma \delta (X)_{0,j-1,0,\ell-1} = 0.
\end{equation*}
\end{theorem}

To prove Theorem \ref{identities}, we will actually prove 
the following generalization, which reduces to Theorem \ref{identities}
when $k=0$.
\begin{theorem}\label{3param}
For every word $X$ in $D$ and $E$, 
and all $j, k, \ell \in \Z_{\geq 0}$, we have
\begin{align*}
\beta (XD)_{0,j,k,\ell} -  \delta & (XE)_{0,j-1,k,\ell} - \alpha \beta (X)_{0,j,k,\ell-1} + q^{|X|+k} \gamma \delta (X)_{0,j-1,k,\ell-1}\\ &=
(1-q) \sum_{a, b \geq 1} q^{a|X|} \cdot \frac{a}{b} \cdot E_{0,a,k,k-b} \cdot (X)_{0,j-a,k-b,\ell-1} \\
&= (1-q) \sum_{a, b \geq 1} q^{a|X|} \cdot \frac{b-a+1}{b} \cdot D_{0,a,k,k-b} \cdot (X)_{0,j-a,k-b,\ell-1}.
\end{align*}
\end{theorem}

We will prove  Theorem \ref{3param}.
by induction on the length $|X|$ of the word $X$.
Before we begin, we first define a \emph{special  column} in a staircase 
tableau.

\begin{definition}\label{special}
A column in a staircase tableau is \emph{special} if it has
a $\beta$ on the bottom, and the next Greek letter which 
appears above it is a $\delta$.  (In particular, a special column
has at least one $\delta$.)
\end{definition}

We also define the notation
$F_{i,j,k,\ell}=
\begin{cases}
\frac{j-i}{k-\ell} \cdot E_{i,j,k,\ell} &\mbox{if }\ell<k\\
0  &\mbox{otherwise.}
\end{cases}$

It is then easy to prove the following.
\begin{lemma}\label{lem:F}
$F_{i,j,k,\ell}$ is the generating function for the weights
of all possible new special columns
that we could add to the left of a staircase tableau with
$i$ rows indexed by $\delta$ and $k$ rows indexed by $\alpha$ or $\gamma$,
obtaining a new staircase tableau which has $j$ rows indexed by 
$\delta$ and $\ell$ rows indexed by $\alpha$ or $\gamma$.
\end{lemma}

We now turn to the base case of the induction, which is 
Proposition \ref{lem:base} below (when $i=0$).
In the statement of the proposition, 
the notation $\indicator$ is the indicator
function whose value is $1$ if its argument is true, and $0$ otherwise.

\begin{proposition}\label{lem:base}
The following identity holds for all non-negative $i, j, k, \ell$.
\begin{equation}\label{base1}
\beta D_{i,j,k,\ell} - \delta   E_{i,j-1,k,\ell} - q^i \alpha \beta
\indicator_{i=j} \indicator_{\ell=k+1} + q^{i+k} \gamma \delta
\indicator_{j=i+1} \indicator_{\ell=k+1} 
\end{equation}
\begin{align}
&= (1-q) \indicator_{j>i} \indicator_{k-\ell \geq j-i-1} \cdot
\frac{j-i}{k-\ell+1} \cdot E_{i,j,k,\ell-1} \label{base2} \\
&= (1-q) \indicator_{j>i} \indicator_{k-\ell \geq j-i-1} \cdot
\frac{k-\ell-j+i+2}{k-\ell+1} \cdot D_{i,j,k,\ell-1} \label{base3}
\end{align}
\end{proposition}

\begin{proof}
One may prove the proposition by computing
the generating functions for $D_{i,j,k,\ell}$ and 
$E_{i,j,k,\ell}$, using e.g. the techniques from the proof of 
Proposition \ref{prop:maple}.  One may also give a combinatorial
proof, which we will illustrate here.

First note that 
Lemma \ref{lem:base} is obvious when $l=k+1$.  Otherwise, we may 
assume that $\ell \leq k$, in which case
$D_{i,j,k,\ell} = D^{\delta}_{i,j,k,\ell}$ and
$E_{i,j-1,k,\ell} = E^{\beta}_{i,j-1,k,\ell}$.
Therefore in order to show that \eqref{base1} equals \eqref{base2}, we need
to show that 
$\beta D^{\delta}_{i,j,k,\ell} - \delta E^{\beta}_{i,j-1,k,\ell} 
= (1-q) F_{i,j,k,\ell-1}.$  The top and bottom rows of Figure 
\ref{Basefigure} represent the quantities
$\beta D^{\delta}_{i,j,k,\ell} - \delta E^{\beta}_{i,j-1,k,\ell}$ and
$(1-q) F_{i,j,k,\ell-1},$ respectively.
\begin{figure}[h]
\input{Basefigure.pstex_t} 
\caption{}
\label{Basefigure}
\end{figure}

We interpret 
$\beta D^{\delta}_{i,j,k,\ell}$
as the generating function for columns of type
$D^{\delta}_{i,j,k,\ell}$ with an extra $\beta$ added at the bottom.
In other words, these are columns of height
$k+2$, which contain $j-i$ $\delta$'s, 
$(k-\ell)-(j-i)+1$ $\beta$'s, which have a $\beta$
at the very bottom (in the first box) and a $\delta$ just above it
(in the second box). The third box will contain either a Greek 
letter (represented by $G$ in the figure) or will be blank, in which 
case it gets the weight $u$.
Similarly, we interpret $\delta E^{\beta}_{i,j-1,k,\ell}$ as 
the generating function for columns of height $k+2$,
which contain $(j-1)-i+1$ $\delta$'s, 
$(k-l)-(j-i-1)$ $\beta$'s, which have a $\delta$ in the first box,
and a $\beta$ in the second box.  The third box will contain either
a Greek letter, or will be blank, in which case it gets the weight $q$.
As illustrated in Figure \ref{Basefigure}, 
the generating functions for the two sets of columns 
which have a Greek letter as their third box are equal, 
and hence $\beta D^{\delta}_{i,j,k,\ell} - \delta E^{\beta}_{i,j-1,k,\ell}$ 
represents the signed union of the remaining columns whose third box
is blank.

We now consider the quantity 
$F_{i,j,k,\ell-1}-q F_{i,j,k,\ell-1}$. 
$F_{i,j,k,\ell-1}$ is the generating function for columns of
height $k+1$, which contain $j-i$ $\delta$'s, 
$k-\ell+1-(j-i)$ $\beta$'s, which have a $\beta$ in the first box,
and whose next Greek letter above the $\beta$ is a $\delta$.
We insert a new blank box (with weight $u$) above the $\delta$, 
so as to interpret $F_{i,j,k,\ell-1}$ as a generating function for
columns of height $k+2$.
Inserting a new blank box just above the $\beta$ in the bottom box,
we interpret $q F_{i,j,k,\ell-1}$ as the generating function for 
columns of height $k+2$, which contain $j-i$ $\delta$'s, 
$k-\ell+1-(j-i)$ $\beta$'s, which have a $\beta$ in the first 
box,
then at least one blank box (with weight $q$) above that,
and whose next Greek letter above the $\beta$ is a $\delta$.
We partition the set of columns enumerated by $F_{i,j,k,\ell-1}$ 
into two parts, based on whether there is at least one blank box
above the bottom $\beta$ or not.  And we partition the set of columns
enumerated by $qF_{i,j,k,\ell-1}$ into two parts, based on whether 
the $\delta$ has a blank box above it or a Greek letter above it.
As illustrated in the figure, we get a cancellation among two 
of the four parts.  Finally note that the set of columns 
with a $\beta$ in the first box, one or more blank boxes above the 
$\beta$, a $\delta$ above that, and a Greek letter above the 
$\delta$, have a weight-preserving bijection with the set of columns  
with a $\delta$ in the first box, a $\beta$ in the second box,
and a blank box (with weight $q$) in the third box.
(Simply move the $\delta$ from the first set of columns to the bottom.)
Therefore $\beta D^{\delta}_{i,j,k,\ell} - \delta E^{\beta}_{i,j-1,k,\ell}=
(1-q) F_{i,j,k,\ell-1}$.

One may give a similar proof that \eqref{base1} equals \eqref{base3},
or alternatively check directly that \eqref{base2} equals \eqref{base3}.
This completes the proof of the proposition.
\end{proof}

The inductive step of the proof of Theorem \ref{3param} has two 
cases: we need to prove that    
Theorem \ref{3param} holds for words of the form $EY$ and also $DY$,
where $Y$ is a word in 
$D$ and $E$, given that Theorem \ref{3param} holds for the word $Y$.
Our strategy is to explicitly multiply $E$ and $Y$
(respectively $D$ and $Y$),
expressing $(EY)_{0,j,k,\ell}$ (respectively $(DY)_{0,j,k,\ell}$)
in terms of quantities
such as $(Y)_{0,a,b,c}$.  
The argument is similar for both cases, so for the sake of brevity,
we will include only the proof in the first case.
To multiply $E$ and $Y$, we use the following lemma, which 
follows easily from Lemma \ref{decrease} and
the definition of our matrices.
\begin{lemma}\label{lem:multiply}
For any word $Z$ in $D$ and $E$ and any non-negative
$j, k, \ell$, we have
$(EZ)_{0,j,k,\ell} = q^k \gamma (Z)_{0,j,k+1,\ell} + 
\sum_{r,t \geq 0} q^{r|Z|} E_{0,r,k,k-t} (Z)_{0,j-r,k-t,\ell}.$
\end{lemma}

To treat the inductive step,
we need to show that the quantity 
\begin{align*}
\beta (EXD)_{0,j,k,\ell} -  & \delta  (EXE)_{0,j-1,k,\ell} 
- \alpha \beta (EX)_{0,j,k,\ell-1} + 
q^{|X|+1+k} \gamma \delta (EX)_{0,j-1,k,\ell-1}\\ - &
(1-q) \sum_{a, b \geq 1} q^{a(|X|+1)} \cdot \frac{a}{b} \cdot E_{0,a,k,k-b} 
\cdot (EX)_{0,j-a,k-b,\ell-1}
\end{align*}
equals $0$.
We start by applying
Lemma \ref{lem:multiply} to the quantities
$(EXD)_{0,j,k,\ell}$, $(EXE)_{0,j-1,k,\ell}$,
$(EX)_{0,j,k,\ell-1}$, $(EX)_{0,j-1,k,\ell-1}$,
and $(EX)_{0,j-a,k-b,\ell-1}$.
We obtain
\begin{align*}
 \beta  &\Bigl( q^k \gamma  (XD)_{0,j,k+1,\ell} 
  + \sum_{r,t \geq 0} q^{r(|X|+1)} E_{0,r,k,k-t} (XD)_{0,j-r,k-t,\ell}\Bigr)\\
-  \delta  & \Bigl(q^k \gamma  (XE)_{0,j-1,k+1,\ell} + 
\sum_{r,t \geq 0} q^{r(|X|+1)} E_{0,r,k,k-t} (XE)_{0,j-r-1,k-t,\ell} \Bigr)\\
- \alpha \beta  & \Bigl(q^k  \gamma (X)_{0,j,k+1,\ell-1} + 
 \sum_{r,t \geq 0} q^{r|X|} E_{0,r,k,k-t} (X)_{0,j-r,k-t,\ell-1} \Bigr)\\
+ \gamma \delta &  q^{|X|+1+k}  \Bigl(q^k \gamma (X)_{0,j-1,k+1,\ell-1} + 
\sum_{r,t \geq 0} q^{r|X|} E_{0,r,k,k-t} (X)_{0,j-r-1,k-t,\ell-1} \Bigr)\\
- & (1-q)  \sum_{a,b \geq 1}  q^{a(|X|+1)} \frac{a}{b} E_{0,a,k,k-b} \Bigl(q^{k-b} \gamma (X)_{0,j-a,k+1-b,\ell-1} + \\
  & \hspace{2.2in}  \sum_{r,t \geq 0} q^{r|X|} E_{0,r,k-b,k-t-b} (X)_{0,j-r-a,k-t-b,\ell-1} \Bigr).
\end{align*}

We now apply the inductive hypothesis several times, to rewrite
$\beta(XD)_{0,j,k+1,\ell}$, as well as 
$\beta(XD)_{0,j-r,k-t,\ell}$ for all $r,t$.
We obtain the expression
\begin{align*}
& q^k \gamma  \Bigl( (1-q)  \sum_{a,b \geq 1} q^{a|X|} (X)_{0,j-a,k+1-b,\ell-1} 
\cdot \frac{a}{b} E_{0,a,k+1,k+1-b} \Bigr) \\
+  & \sum_{r,t \geq 0} q^{r(|X|+1)}  E_{0,r,k,k-t} 
 \Bigl( \alpha \beta (X)_{0,j-r,k-t,\ell-1} - q^{|X|+k-t} \gamma \delta
(X)_{0,j-1-r,k-t,\ell-1}+  \\
 & \hspace{1.6in}   (1-q) \sum_{a,b \geq 1} q^{a|X|} (X)_{0,j-a-r,k-t-b,\ell-1} 
\cdot \frac{a}{b} E_{0,a,k-t,k-t-b} \Bigr)\\
- & \alpha \beta \sum_{r,t \geq 0} q^{r|X|}  E_{0,r,k,k-t} 
(X)_{0,j-r,k-t,\ell-1}
+ q^{|X|+1+k} \gamma \delta   \sum_{r,t \geq 0}  q^{r|X|} E_{0,r,k,k-t} 
(X)_{0,j-r-1,k-t,\ell-1} \\
- & (1-q) \sum_{a,b \geq 1} q^{a (|X|+1)} \frac{a}{b}  E_{0,a,k,k-b} 
   \Bigl(q^{k-b} \gamma (X)_{0,j-a,k+1-b,\ell-1} + \\
   & \hspace{2.3in}  \sum_{r,t \geq 0} q^{r|X|} E_{0,r,k-b,k-t-b} (X)_{0,j-r-a,k-t-b,\ell-1}
    \Bigr). 
\end{align*}
The expression above may be viewed as a linear combination of 
terms $(X)_{0,v,w,\ell-1}$, where $v \leq j$, and $w \leq k$.
To prove that the expression is identically $0$, we will show
that the coefficient of each such $(X)_{0,v,w,\ell-1}$ is $0$.

First note that the coefficient of $(X)_{0,j,k-y,\ell-1}$
above (where $y \geq 0$) is 
$E_{0,0,k,k-y} \alpha \beta - \alpha \beta E_{0,0,k,k-y} = 0$.
Therefore it suffices to analyze the coefficient of 
each $(X)_{0,j-1-x,k-y,\ell-1}$ for all $x,y \geq 0$.
This coefficient is:
\begin{align}
&  q^k \gamma (1-q) q^{(x+1)|X|}\cdot  \frac{x+1}{y+1}\cdot E_{0,x+1,k+1,k-y}\label{eq1} \\
+ &  q^{(x+1)(|X|+1)} \alpha \beta E_{0,x+1,k,k-y} \label{eq2} \\
- & q^{(x+1)|X|+x+k-y} \gamma \delta E_{0,x,k,k-y} \label{eq3}\\
+ & \sum_{r,t \geq 0} q^{(x+1)|X|+r} (1-q) E_{0,r,k,k-t} 
\cdot \frac{x+1-r}{y-t} \cdot E_{0,x+1-r,k-t,k-y} \label{eq4}\\
- & \alpha \beta q^{(x+1)|X|} E_{0,x+1,k,k-y}\label{eq5} \\
+ & q^{(x+1)|X|+1+k} \gamma \delta E_{0,x,k,k-y} \label{eq6}\\
- & q^{(x+1)|X|+x+k-y} (1-q) \gamma\cdot \frac{x+1}{y+1}\cdot E_{0,x+1,k,k-y-1}\label{eq7}\\
- & (1-q) \sum_{r,t \geq 0} q^{(x+1)(|X|+1)-r} E_{0,r,k-y+t,k-y} 
\cdot \frac{x+1-r}{y-t} \cdot E_{0,x+1-r,k,k-y+t}.\label{eq8}
\end{align}

Note that every term above contains a factor of 
$q^{(x+1)|X|}$, which we may delete (since our goal is to show
that the sum of the terms is $0$).
If we then combine 
\eqref{eq2} and \eqref{eq5},
\eqref{eq3} and \eqref{eq6},
\eqref{eq1} and \eqref{eq4} (using Lemma \ref{decrease}),
and \eqref{eq7} and \eqref{eq8} (using Lemma \ref{decrease}),
we get
\begin{align}
& (q^{x+1}-1) \alpha \beta E_{0,x+1,k,k-y} \label{eq9} \\
+ & (q^{k+1}-q^{k+x-y}) \gamma \delta E_{0,x,k,k-y} \label{eq10}\\
+ &  (1-q) \sum_{r \geq 0, t \geq -1} E_{0,r,k,k-t}
\cdot\frac{x+1-r}{y-t} \cdot E_{r,x+1,k-t,k-y} \label{eq11}\\
- &  (1-q) \sum_{r \geq 0, t \geq -1} \frac{x+1-r}{y-t} \cdot
 E_{0,x+1-r,k,k-y+t} E_{x+1-r,x+1,k-y+t,k-y}.\label{eq12}
\end{align}

To complete the proof, it suffices to show that the above sum
is $0$.

Using the notation $F_{i,j,k,\ell}$ defined earlier, to show that the sum of 
\eqref{eq9}, \eqref{eq10}, \eqref{eq11}, and \eqref{eq12} vanishes,
one may equivalently show the following identity.

\begin{proposition}\label{prop:maple}
For all non-negative $x$ and $y$, we have that
\small{
\begin{equation*}
(1-q)(EF-FE)_{0,x+1,k,k-y}
= (1-q^{x+1}) \alpha \beta E_{0,x+1,k,k-y} + 
(q^{k+x-y} -q^{k+1}) \gamma \delta E_{0,x,k,k-y}.
\end{equation*}
}
\end{proposition}

\begin{proof}
By Lemma \ref{lem:F}, the equation in 
Proposition \ref{prop:maple} has a combinatorial interpretation 
in terms of columns (and pairs of columns) of staircase tableaux.
We prove the equation by computing the 
generating functions for such columns.

Define the following generating functions:
\begin{align*}
\mathcal{E}(r,z,s) &= \sum_{j, k, \ell} E_{0,j,k,\ell}\ r^j z^k s^{\ell},\\
\mathcal{EF}(r,z,s) &= \sum_{j,k,\ell} (EF)_{0,j,k,\ell} r^j z^k s^{\ell},\\
\mathcal{FE}(r,z,s) &= \sum_{j,k,\ell} (FE)_{0,j,k,\ell} r^j z^k s^{\ell}.
\end{align*}
where the sum is over all non-negative $j, k, \ell$ such that $\ell \leq k$.
Then in order to prove 
Proposition \ref{prop:maple}, we need to verify that 
{\small
\begin{equation*}\label{eq:maple}
(1-q) (\mathcal{EF}(r,z,s) - \mathcal{FE}(r,z,s)) = 
\alpha \beta \mathcal{E}(r,z,s) - \alpha \beta \mathcal{E}(qr,z,s)
+\gamma \delta r \mathcal{E}(qr, z, qs) -\gamma \delta qr
\mathcal{E}(r,qz, s).
\end{equation*}
}

One may compute the three generating functions explicitly
by hand.  To compute $\mathcal{E}(r,z,s)$, note that 
since the sum is over $j, k, \ell$ where $\ell \leq k$,
we need to enumerate columns with a $\beta$ at the bottom.
One may construct such a column from bottom to top:
above the $\beta$, there is a non-negative number of blank boxes
each with weight $q$.  Above these there is an arbitrary 
sequence of $\beta$'s, blank boxes, and $\delta$'s, which we 
partition into blocks consisting of a $\beta$ at the bottom
with some $q$'s above it, and blocks consisting of a 
$\delta$ at the bottom with some $u$'s above it.  Above 
these blocks, there is either an alpha with some $u$'s above it,
or a $\gamma$ with some $q$'s above it, or nothing.  Therefore,
since we set $u=1$, we have
$$\mathcal{E}(r,z,s) = 
\beta  \cdot \frac{1}{1-qzs}\cdot \frac{1}{1-(\frac{\beta z}{1-qzs} + \frac{\delta r z}{1-zs})}\cdot \left(\frac{\alpha z s}{1-zs} + \frac{\gamma zs}{1-qzs} + 1\right).$$

The expressions for the generating functions 
$\mathcal{EF}(r,z,s)$ and $\mathcal{FE}(r,z,s)$ 
are quite complicated, 
so we provide 
a Maple worksheet to compute them and to 
check the identity relating 
$\mathcal{E}(r,z,s)$, 
$\mathcal{EF}(r,z,s)$ and $\mathcal{FE}(r,z,s)$. 
The Maple worksheet may be downloaded at 
\texttt{www.math.berkeley.edu/$\sim$williams/papers/CW-Identity.zip}.
\end{proof}

This completes the proof of Theorem \ref{3param}.

\begin{remark}
It would be interesting to find a combinatorial proof of 
Proposition \ref{prop:maple}.
\end{remark}

\subsection{Applications}\label{Applications}

Once we have a solution to the Matrix Ansatz, it is easy to 
express physical quantities in terms of matrix products
\cite{Derrida1}.  
Set $C=D+E$.

The partition function $Z_n$ is written as 
$ W C^n V,$ and 
the average particle number at site $i$,
$\langle \tau_i \rangle_n$ (where the bracket indicates the average
over the stationary probability distribution) is written as
\begin{eqnarray}
\langle \tau_i\rangle =
\frac{W C^{i-1}DC^{n-i} V}{Z_n}.
\label{def:1Pfn}
\end{eqnarray}
Similarly the two-point function $\langle \tau_i\tau_j\rangle_n$ is given by
\begin{eqnarray}
\langle \tau_i\tau_j\rangle =
\frac{ W C^{i-1}DC^{j-i-1}DC^{n-j} V}{Z_n}, 
\label{def:2Pfn}
\end{eqnarray}
and the $n$-point functions are expressed similarly.
The particle current through the bond between the neighboring sites
from left to right,
which is defined by
$J=\langle \tau_i(1-\tau_{i+1})-q(1-\tau_i)\tau_{i+1}\rangle$,
is simply given by
$J=\frac{Z_{n-1}}{Z_n}.$
This expression is independent of $i$, as expected in 
the steady state.

Note that the matrices $D$ and $E$ that we have defined 
in Section \ref{def-matrices} actually satisfy the Generalized Matrix Ansatz,
not the Matrix Ansatz.  However, we can compare quantities computed
via the two different Ansatzes using Lemma \ref{ansatz-relation}, and 
in particular
equation (\ref{2partition}).
Theorem \ref{Physical} now follows from Theorem \ref{comb-interpret}
and the expressions above for the current and $m$-point
functions in terms of matrix products.

\section{The proof of our Askey-Wilson moment formula}\label{AWproof}

Before proving Theorem \ref{moments}, we need to prove the following
result.

\begin{lemma}\label{ansatz-relation}
Let $D, E, W, V$ be a solution to the Ansatz of Theorem \ref{ansatz2},
and let $\tilde{D}, \tilde{E}, \tilde{W}, \tilde{V}$ be a solution
to the Ansatz of Theorem \ref{ansatz}.   Let $h$ denote the ratio
$\frac{\tilde{W} \tilde{V}}{WV}$. Then if $X$ is a word in 
$D$ and $E$, and $\tilde{X}$ is the corresponding word in 
$\tilde{D}$ and $\tilde{E}$, then 
$$WXV = h^{-1} \tilde{W} \tilde{X} \tilde{V}  \prod_{i=0}^{|X|-1} \lambda_i.$$ 

In particular, if 
$\tilde{Z}_n = \tilde{W} (\tilde{D}+\tilde{E})^n \tilde{V}$,
then  
$Z_n = h^{-1} \tilde{Z}_n \prod_{i=0}^{n-1} \lambda_i.$
\end{lemma}

\begin{proof}
Let $\tau$ denote the type of $X$, and let $n=|X|$.  We use induction on $n$.
By Theorem \ref{ansatz2} and Theorem \ref{ansatz} respectively,
$WXV$ and $\tilde{W} \tilde{X} \tilde{V}$ compute
(unnormalized) steady state probabilities of being in state $\tau$.
Therefore $WXV = c_n \tilde{W} \tilde{X} \tilde{V}$ for some constant $c_n$
that depends  on $n$ but not $X$.   We want to show that 
$c_n = h^{-1} \prod_{i=0}^{n-1} \lambda_i$.

Since  we have assumed that $D, E, W, V$ satisfy the relations of Theorem \ref{ansatz2},
$\gamma WDXV - \alpha WEXV = \lambda_{n} WXV$.  By induction, we conclude that
$$\gamma WDXV - \alpha WEXV = 
  \lambda_{n} \tilde{W} \tilde{X} \tilde{V} h^{-1} \prod_{i=0}^{n-1} \lambda_i.$$
But also 
\begin{align*}
\gamma WDXV - \alpha WEXV &= c_{n+1} \gamma \tilde{W}\tilde{D}\tilde{X}\tilde{V} - c_{n+1} \alpha \tilde{W} \tilde{E} \tilde{X} \tilde{V} \\
   &= c_{n+1} (\gamma \tilde{W}\tilde{D}\tilde{X} \tilde{V} - 
               \alpha \tilde{W} \tilde{E} \tilde{X} \tilde{V}) \\
   &= c_{n+1} \tilde{W} \tilde{X} \tilde{V},
\end{align*}
by Theorem \ref{ansatz}.
This shows
that $c_{n+1} = h^{-1} \prod_{i=0}^{n} \lambda_i,$ 
which completes the proof.
\end{proof}

We now prove Theorem \ref{moments}, using some results of 
\cite{USW}.

\begin{proof}[Proof of Theorem \ref{moments}]
Let $\tilde{Z}_n$ denote the partition function from \cite{USW},
i.e. $\tilde{Z}_n = \tilde{W} (\tilde{D}+\tilde{E})^n \tilde{V}$,
where $\tilde{D}, \tilde{E}, \tilde{W}, \tilde{V}$ are a solution
to the Ansatz of 
Theorem \ref{ansatz}, and $\tilde{W} \tilde{V}=h_0$ 
(see \cite[(4.18)]{USW}).  Here $h_0$ is as in Section 
\ref{AW Section}.
Then by \cite[Section 6.1]{USW},
\begin{equation*}
\tilde{Z}_n = \oint_C\frac{dz}{4\pi iz}
w\left((z+z^{-1})/2\right)
\left[ \frac{z+z^{-1}+2}{1-q}\right] ^n.
\end{equation*}
Therefore
\begin{equation*}
\tilde{Z}_n = \oint_C\frac{dz}{4\pi iz}
w\left((z+z^{-1})/2\right)\left(\frac{2}{1-q}\right)^n
\left[ \frac{z+z^{-1}}{2} + 1\right] ^n,
\end{equation*}
which implies that 
\begin{align*}
\left(\frac{1-q}{2}\right)^n \tilde{Z}_n &= \oint_C\frac{dz}{4\pi iz}
w\left((z+z^{-1})/2\right)
\left[ \frac{z+z^{-1}}{2} + 1\right] ^n\\
 &= \sum_{k=0}^n {n \choose k} 
 \oint_C\frac{dz}{4\pi iz}
w\left((z+z^{-1})/2\right)
\left[ \frac{z+z^{-1}}{2}\right] ^k\\
&= \sum_{k=0}^n {n \choose k} \mu_k.
\end{align*}

Inverting this, we get 
\begin{equation*}
\mu_k = \sum_{n=0}^k (-1)^{k-n} {k\choose n}
   \left(\frac{1-q}{2}\right)^{n} \tilde{Z}_{n}.
\end{equation*}

By Theorem \ref{NewThm}, we know that 
$Z_{n}$ is the generating function for all staircase
tableaux of size $n$.  
By Lemma \ref{ansatz-relation},
\begin{equation} \label{2partition}
Z_n = h_0^{-1} \tilde{Z}_n \prod_{i=0}^{n-1} (\alpha \beta - \gamma \delta q^i).
\end{equation}
Therefore $$\mu_k = h_0 \sum_{n=0}^k (-1)^{k-n}{k \choose n} 
\left(\frac{1-q}{2}\right)^{n}
 \frac{{Z}_{n}}{\prod_{i=0}^{n-1} (\alpha \beta - \gamma \delta q^i)}.$$
\end{proof}

\section{Open problems}\label{conc}

\subsection{Symmetries in the ASEP}
Recall that the ASEP has 
``left-right," ``arrow-reversal,"
and ``particle-hole" symmetries, which
imply 
Observation \ref{symmetries}.

\begin{problem}\label{sym}
For each symmetry above, prove the corresponding identity 
in Observation \ref{symmetries} by describing an appropriate involution
on staircase tableaux.
\end{problem}

Proving the second identity  
in this manner is easy.
Namely, define
a map $\iota$  by
letting $\iota(\T)$ be the tableau obtained from $\T$
by switching $\beta$'s and $\delta$'s,
and switching $\alpha$'a and $\gamma$'s; clearly
if $\wt(\T) = \alpha^{i_1} \beta^{i_2} \gamma^{i_3} \delta^{i_4} q^{i_5} u^{i_6},$ 
then $\wt(\iota(\T)) = 
\alpha^{i_3} \beta^{i_4} \gamma^{i_1} \delta^{i_2} q^{i_6} u^{i_5}.$  This 
plus Theorem \ref{NewThm}
proves the second identity. 
It remains to find an involution $\iota'$ proving the first
identity
(the remaining involution can  be 
constructed by composing $\iota'$ with $\iota$).
A natural guess is to define
$\iota'(\T)$ by transposing $\T$ then switching
$\alpha$'s and $\delta$'s, and $\beta$'s and $\gamma$'s.  This works
when $q=u$, but not for $q \neq u$.

\subsection{Lifting the ASEP to a Markov chain on staircase tableaux}
\begin{problem}
Define a Markov chain on the set of all staircase tableaux of size $n$
which \emph{projects} to the ASEP
in the sense of \cite{CW2},  
such that the steady state probability of a tableau $\T$ is 
proportional to $\wt(T)$.
Such an approach would give a completely
combinatorial proof of Theorem \ref{NewThm}.
(This was done in \cite{CW2} for $\gamma=\delta=0$.)
\end{problem}

\subsection{``Birth certificates" for particles}

In the ASEP, a black particle enters from either the 
left (at rate $\alpha$) or from the right (at rate $\delta$).
Similarly, a ``hole" (or a white particle) enters from either the 
left (at rate $\gamma$) or from the right (at rate $\beta$).  One could 
imagine defining a more refined ASEP, in which each particle in the lattice
has attached to it its ``birth certificate," that is, the information
of whether it entered the lattice from the left or from the right.
Such an ASEP would be a Markov chain on $4^n$ states (all words of length 
$n$ in $\alpha, \beta, \gamma$ and $\delta$), 
which projects to the ASEP upon mapping the letters $\alpha$ and $\delta$ 
to a black particle, and the letters $\beta$ and $\gamma$ to a white particle.
One could then hope to prove an analogue of
Theorem \ref{NewThm} as follows:

\begin{problem}
Fix a lattice of $n$ sites, and 
let $S$ be the set
of all $4^n$ 
words of length $n$ on the alphabet 
$\{\alpha, \beta, \gamma, \delta\}$, which we think of as 
configurations of four kinds of particles -- two kinds 
of black particles, 
labeled $\alpha$ and $\delta$, and two kinds of white particles,
labeled $\gamma$
and $\beta$.  Define a Markov chain on $S$ 
with the following properties:
\begin{itemize}
\item particles labeled $\alpha$ and $\gamma$ always enter the lattice
  from the left, and particles labeled $\beta$ and $\delta$ always enter from the right;
\item the Markov chain projects to the ASEP;
\item the steady state probability of state $(\tau_1,\dots,\tau_n)$
is proportional to the generating function for all 
staircase tableaux whose border is $(\tau_1,\dots,\tau_n)$.
\end{itemize}
\end{problem}

\subsection{A combinatorial proof of the relations of the Ansatz}

In Section \ref{proof1}, we gave a combinatorial proof that
$D,E,V,W$ satisfy  
relation (III) of 
Theorem \ref{ansatz2}, by translating it into a statement about
tableaux.  However, we have not yet found a combinatorial proof 
that $D,E,V,W$ satisfy (I) and (II).

\begin{problem}
Give a combinatorial proof of relations (I) and (II)
of Theorem 5.2.  
\end{problem}

We note that when $q=u$, or one of 
$\alpha, \beta, \gamma, \delta$ is $0$, the above problem is easy.

\subsection{Specializing our moment formula for Askey-Wilson polynomials}

\begin{problem}
Show directly that our moment formula  recovers
already-known moment formulas for specializations or limiting cases of 
Askey-Wilson polynomials.
\end{problem}

\section{Appendix: Staircase, permutation, and alternative tableaux}

\begin{definition}\cite{SW, Postnikov}
A \emph{permutation tableau}  $\T$ is a Young diagram 
(where rows may have length $0$) whose
boxes are filled with $0$'s and $1$'s,
such that
each column contains at least one $1$, and there is no $0$
which has simultaneaously a $1$ above it in the same column
and a $1$ to its left in the same row. The \emph{length} of $\T$
is the sum of its number of rows and columns.
\end{definition}

\begin{definition}\cite{Viennot}
An \emph{alternative tableau} $\T$ is a Young diagram 
(where rows \emph{and} columns may have length $0$) whose
boxes are either empty or filled with
left arrows $\leftarrow$ or up arrows $\uparrow$,
such that all boxes to the left of a $\leftarrow$ are empty
and all boxes above an $\uparrow$ are empty.\footnote{Actually
the alternative tableaux of \cite{Viennot} were defined 
as Young diagrams with blue, red
and empty boxes; we define them here using left and up arrows,
instead of blue and red boxes, following 
\cite{nadeau}.}
The \emph{length} of $\T$
is the sum of its number of rows and columns.
\end{definition}

See  \cite{nadeau} for more information about alternative tableaux.

\begin{proposition}
There is a bijection between staircase tableaux of size
$n$ which do not contain any $\gamma$ or $\delta$, and:
\begin{enumerate}
\item permutation tableaux of length $n+1$;
\item alternative tableaux of length $n$.
\end{enumerate}
\end{proposition}

\begin{proof}
We first give a bijection from permutation tableaux to staircase tableaux.
Define a {\it restricted $0$} of a permutation tableau to be 
a $0$ which has a $1$ above it in the same column.  A restricted $0$
is {\it rightmost} if it is the rightmost restricted $0$ in its row.
If $\T$ is a permutation tableau, we replace with a $\leftarrow$
every rightmost restricted
$0$, and replace with a $\uparrow$ 
every $1$ which is the highest $1$
in its column but is not in the top row.  We replace every other
entry of $\T$ by an empty box, and delete the top row
(but we remember the length of the top row by possibly
inserting empty columns to the right).  The result
is an alternative tableau, 
see  Figure \ref{ST-PT}, and the map can be easily inverted.\\

For the second bijection, fix a staircase tableau of size $n$.
For $i$ from $1$ to $n$, if the $i$th diagonal box contains an $\alpha$,
then delete this entry and the column above it
(this $\alpha$ will correspond to a vertical step in the 
south-east border of the resulting alternative tableau).  Otherwise
if the $i$th diagonal box contains a $\beta$, delete this entry and 
the row to its left (this $\beta$ will correspond to a horizontal step
in the south-east border of the resulting tableau).
Then replace each $\alpha$ with an $\uparrow$ and each $\beta$ with
a $\leftarrow$, and discard all the other entries.
See Figure \ref{ST-PT}.
\end{proof}

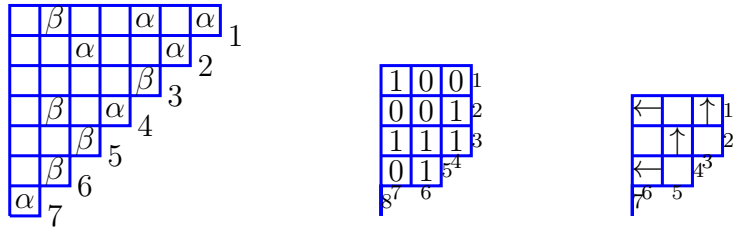
\begin{figure}[htp]
\centering
\psset{unit=0.4cm}\psset{linewidth=0.4mm}
\begin{pspicture}(0,0)(7,7)
\psline(0,0)(1,0)(1,1)(2,1)(2,2)(3,2)(3,3)(4,3)(4,4)
(5,4)(5,5)(6,5)(6,6)(7,6)(7,7)(0,7)(0,0)
\psline(1,1)(1,7)
\psline(2,2)(2,7)
\psline(3,3)(3,7)
\psline(4,4)(4,7)
\psline(5,5)(5,7)
\psline(6,6)(6,7)
\psline(1,1)(0,1)
\psline(2,2)(0,2)
\psline(3,3)(0,3)
\psline(4,4)(0,4)
\psline(5,5)(0,5)
\psline(6,6)(0,6)
\rput(0.5,0.5){$\alpha$}
\rput(1.5,1.5){$\beta$}
\rput(2.5,2.5){$\beta$}
\rput(3.5,3.5){$\alpha$}
\rput(4.5,4.5){$\beta$}
\rput(5.5,5.5){$\alpha$}
\rput(6.5,6.5){$\alpha$}
\rput(1.5,3.5){$\beta$}
\rput(2.5,5.5){$\alpha$}
\rput(4.5,6.5){$\alpha$}
\rput(1.5,6.5){$\beta$}
\rput(1.5,0){7}
\rput(2.5,1){6}
\rput(3.5,2){5}
\rput(4.5,3){4}
\rput(5.5,4){3}
\rput(6.5,5){2}
\rput(7.5,6){1}
\end{pspicture}\hspace{2cm}
\begin{pspicture}(0,0)(3,5)
\psline(0,0)(0,1)(2,1)(2,2)(3,2)(3,5)(0,5)(0,0)
\psline(1,1)(1,5)
\psline(2,2)(2,5)
\psline(0,2)(2,2)
\psline(0,3)(3,3)
\psline(0,4)(3,4)
\rput(0.5,1.5){$0$}
\rput(1.5,1.5){$1$}
\rput(0.5,2.5){$1$}
\rput(1.5,2.5){$1$}
\rput(2.5,2.5){$1$}
\rput(0.5,3.5){$0$}
\rput(1.5,3.5){$0$}
\rput(2.5,3.5){$1$}
\rput(0.5,4.5){$1$}
\rput(1.5,4.5){$0$}
\rput(2.5,4.5){$0$}
\rput(3.2,4.5){\tiny 1}
\rput(3.2,3.5){\tiny 2}
\rput(3.2,2.5){\tiny 3}
\rput(2.5,1.8){\tiny 4}
\rput(2.2,1.5){\tiny 5}
\rput(1.5,0.8){\tiny 6}
\rput(0.5,0.8){\tiny 7}
\rput(0.2,0.5){\tiny 8}
\end{pspicture}\hspace{2cm}
\begin{pspicture}(0,0)(3,5)
\psline(0,0)(0,1)(2,1)(2,2)(3,2)(3,4)(0,4)(0,0)
\psline(1,1)(1,4)
\psline(2,2)(2,4)
\psline(0,2)(2,2)
\psline(0,3)(3,3)
\rput(0.5,1.5){$\leftarrow$}
\rput(0.5,3.5){$\leftarrow$}
\rput(1.5,2.5){$\uparrow$}
\rput(2.5,3.5){$\uparrow$}
\rput(3.2,3.5){\tiny 1}
\rput(3.2,2.5){\tiny 2}
\rput(2.5,1.8){\tiny 3}
\rput(2.2,1.5){\tiny 4}
\rput(1.5,0.8){\tiny 5}
\rput(0.5,0.8){\tiny 6}
\rput(0.2,0.5){\tiny 7}
\end{pspicture}
\caption{From a staircase tableau, to a permutation tableau and an alternative tableau}
\label{ST-PT}
\end{figure}

\end{document}

%% file: Staircase.pstex_t
\begin{picture}(0,0)%
\includegraphics{Staircase.pstex}%
\end{picture}%
\setlength{\unitlength}{1579sp}%
\begingroup\makeatletter\ifx\SetFigFont\undefined%
\gdef\SetFigFont#1#2#3#4#5{%
  \reset@font\fontsize{#1}{#2pt}%
  \fontfamily{#3}\fontseries{#4}\fontshape{#5}%
  \selectfont}%
\fi\endgroup%
\begin{picture}(4687,4611)(3279,-5250)
\put(4651,-1711){\makebox(0,0)[lb]{\smash{{\SetFigFont{14}{16.8}{\rmdefault}{\bfdefault}{\updefault}{\color[rgb]{0,0,0}$\alpha$}%
}}}}
\put(4051,-1111){\makebox(0,0)[lb]{\smash{{\SetFigFont{14}{16.8}{\rmdefault}{\bfdefault}{\updefault}{\color[rgb]{0,0,0}$\beta$}%
}}}}
\put(6451,-1711){\makebox(0,0)[lb]{\smash{{\SetFigFont{14}{16.8}{\rmdefault}{\bfdefault}{\updefault}{\color[rgb]{0,0,0}$\gamma$}%
}}}}
\put(5851,-2311){\makebox(0,0)[lb]{\smash{{\SetFigFont{14}{16.8}{\rmdefault}{\bfdefault}{\updefault}{\color[rgb]{0,0,0}$\delta$}%
}}}}
\put(5251,-2911){\makebox(0,0)[lb]{\smash{{\SetFigFont{14}{16.8}{\rmdefault}{\bfdefault}{\updefault}{\color[rgb]{0,0,0}$\alpha$}%
}}}}
\put(4651,-3511){\makebox(0,0)[lb]{\smash{{\SetFigFont{14}{16.8}{\rmdefault}{\bfdefault}{\updefault}{\color[rgb]{0,0,0}$\delta$}%
}}}}
\put(4051,-4111){\makebox(0,0)[lb]{\smash{{\SetFigFont{14}{16.8}{\rmdefault}{\bfdefault}{\updefault}{\color[rgb]{0,0,0}$\beta$}%
}}}}
\put(3451,-4711){\makebox(0,0)[lb]{\smash{{\SetFigFont{14}{16.8}{\rmdefault}{\bfdefault}{\updefault}{\color[rgb]{0,0,0}$\gamma$}%
}}}}
\put(4051,-2911){\makebox(0,0)[lb]{\smash{{\SetFigFont{14}{16.8}{\rmdefault}{\bfdefault}{\updefault}{\color[rgb]{0,0,0}$\delta$}%
}}}}
\put(7051,-1111){\makebox(0,0)[lb]{\smash{{\SetFigFont{14}{16.8}{\rmdefault}{\bfdefault}{\updefault}{\color[rgb]{0,0,0}$\gamma$}%
}}}}
\put(5851,-1111){\makebox(0,0)[lb]{\smash{{\SetFigFont{14}{16.8}{\rmdefault}{\bfdefault}{\updefault}{\color[rgb]{0,0,0}$\alpha$}%
}}}}
\end{picture}%

%% file: Staircase2.pstex_t
\begin{picture}(0,0)%
\includegraphics{Staircase2.pstex}%
\end{picture}%
\setlength{\unitlength}{1579sp}%
\begingroup\makeatletter\ifx\SetFigFont\undefined%
\gdef\SetFigFont#1#2#3#4#5{%
  \reset@font\fontsize{#1}{#2pt}%
  \fontfamily{#3}\fontseries{#4}\fontshape{#5}%
  \selectfont}%
\fi\endgroup%
\begin{picture}(4687,4611)(3279,-5250)
\put(7051,-1111){\makebox(0,0)[lb]{\smash{{\SetFigFont{14}{16.8}{\rmdefault}{\bfdefault}{\updefault}{\color[rgb]{0,0,0}$\gamma$}%
}}}}
\put(6451,-1711){\makebox(0,0)[lb]{\smash{{\SetFigFont{14}{16.8}{\rmdefault}{\bfdefault}{\updefault}{\color[rgb]{0,0,0}$\gamma$}%
}}}}
\put(5851,-2311){\makebox(0,0)[lb]{\smash{{\SetFigFont{14}{16.8}{\rmdefault}{\bfdefault}{\updefault}{\color[rgb]{0,0,0}$\delta$}%
}}}}
\put(4651,-3511){\makebox(0,0)[lb]{\smash{{\SetFigFont{14}{16.8}{\rmdefault}{\bfdefault}{\updefault}{\color[rgb]{0,0,0}$\delta$}%
}}}}
\put(4051,-4111){\makebox(0,0)[lb]{\smash{{\SetFigFont{14}{16.8}{\rmdefault}{\bfdefault}{\updefault}{\color[rgb]{0,0,0}$\beta$}%
}}}}
\put(3451,-4711){\makebox(0,0)[lb]{\smash{{\SetFigFont{14}{16.8}{\rmdefault}{\bfdefault}{\updefault}{\color[rgb]{0,0,0}$\gamma$}%
}}}}
\put(5326,-2911){\makebox(0,0)[lb]{\smash{{\SetFigFont{14}{16.8}{\rmdefault}{\bfdefault}{\updefault}{\color[rgb]{0,0,0}$\alpha$}%
}}}}
\put(4051,-4111){\makebox(0,0)[lb]{\smash{{\SetFigFont{14}{16.8}{\rmdefault}{\bfdefault}{\updefault}{\color[rgb]{0,0,0}$\beta$}%
}}}}
\put(4051,-1111){\makebox(0,0)[lb]{\smash{{\SetFigFont{14}{16.8}{\rmdefault}{\bfdefault}{\updefault}{\color[rgb]{0,0,0}$\beta$}%
}}}}
\put(5851,-1111){\makebox(0,0)[lb]{\smash{{\SetFigFont{14}{16.8}{\rmdefault}{\bfdefault}{\updefault}{\color[rgb]{0,0,0}$\alpha$}%
}}}}
\put(4651,-1711){\makebox(0,0)[lb]{\smash{{\SetFigFont{14}{16.8}{\rmdefault}{\bfdefault}{\updefault}{\color[rgb]{0,0,0}$\alpha$}%
}}}}
\put(4051,-2911){\makebox(0,0)[lb]{\smash{{\SetFigFont{14}{16.8}{\rmdefault}{\bfdefault}{\updefault}{\color[rgb]{0,0,0}$\delta$}%
}}}}
\end{picture}%

%% file: Example.pstex_t
\begin{picture}(0,0)%
\includegraphics{Example.pstex}%
\end{picture}%
\setlength{\unitlength}{1381sp}%
\begingroup\makeatletter\ifx\SetFigFont\undefined%
\gdef\SetFigFont#1#2#3#4#5{%
  \reset@font\fontsize{#1}{#2pt}%
  \fontfamily{#3}\fontseries{#4}\fontshape{#5}%
  \selectfont}%
\fi\endgroup%
\begin{picture}(11744,1244)(279,-1883)
\put(1051,-1111){\makebox(0,0)[lb]{\smash{{\SetFigFont{12}{14.4}{\rmdefault}{\mddefault}{\updefault}{\color[rgb]{0,0,0}$\alpha$}%
}}}}
\put(451,-1711){\makebox(0,0)[lb]{\smash{{\SetFigFont{12}{14.4}{\rmdefault}{\mddefault}{\updefault}{\color[rgb]{0,0,0}$\alpha$}%
}}}}
\put(3451,-1711){\makebox(0,0)[lb]{\smash{{\SetFigFont{12}{14.4}{\rmdefault}{\mddefault}{\updefault}{\color[rgb]{0,0,0}$\alpha$}%
}}}}
\put(5551,-1111){\makebox(0,0)[lb]{\smash{{\SetFigFont{12}{14.4}{\rmdefault}{\mddefault}{\updefault}{\color[rgb]{0,0,0}$\alpha$}%
}}}}
\put(6451,-1111){\makebox(0,0)[lb]{\smash{{\SetFigFont{12}{14.4}{\rmdefault}{\mddefault}{\updefault}{\color[rgb]{0,0,0}$\alpha$}%
}}}}
\put(7051,-1111){\makebox(0,0)[lb]{\smash{{\SetFigFont{12}{14.4}{\rmdefault}{\mddefault}{\updefault}{\color[rgb]{0,0,0}$\alpha$}%
}}}}
\put(8551,-1111){\makebox(0,0)[lb]{\smash{{\SetFigFont{12}{14.4}{\rmdefault}{\mddefault}{\updefault}{\color[rgb]{0,0,0}$\alpha$}%
}}}}
\put(10051,-1111){\makebox(0,0)[lb]{\smash{{\SetFigFont{12}{14.4}{\rmdefault}{\mddefault}{\updefault}{\color[rgb]{0,0,0}$\alpha$}%
}}}}
\put(11551,-1111){\makebox(0,0)[lb]{\smash{{\SetFigFont{12}{14.4}{\rmdefault}{\mddefault}{\updefault}{\color[rgb]{0,0,0}$\alpha$}%
}}}}
\put(10951,-1711){\makebox(0,0)[lb]{\smash{{\SetFigFont{12}{14.4}{\rmdefault}{\mddefault}{\updefault}{\color[rgb]{0,0,0}$\delta$}%
}}}}
\put(10951,-1111){\makebox(0,0)[lb]{\smash{{\SetFigFont{12}{14.4}{\rmdefault}{\mddefault}{\updefault}{\color[rgb]{0,0,0}$\delta$}%
}}}}
\put(9451,-1711){\makebox(0,0)[lb]{\smash{{\SetFigFont{12}{14.4}{\rmdefault}{\mddefault}{\updefault}{\color[rgb]{0,0,0}$\delta$}%
}}}}
\put(9451,-1111){\makebox(0,0)[lb]{\smash{{\SetFigFont{12}{14.4}{\rmdefault}{\mddefault}{\updefault}{\color[rgb]{0,0,0}$\gamma$}%
}}}}
\put(7951,-1111){\makebox(0,0)[lb]{\smash{{\SetFigFont{12}{14.4}{\rmdefault}{\mddefault}{\updefault}{\color[rgb]{0,0,0}$\beta$}%
}}}}
\put(7951,-1711){\makebox(0,0)[lb]{\smash{{\SetFigFont{12}{14.4}{\rmdefault}{\mddefault}{\updefault}{\color[rgb]{0,0,0}$\delta$}%
}}}}
\put(6451,-1711){\makebox(0,0)[lb]{\smash{{\SetFigFont{12}{14.4}{\rmdefault}{\mddefault}{\updefault}{\color[rgb]{0,0,0}$\delta$}%
}}}}
\put(4951,-1711){\makebox(0,0)[lb]{\smash{{\SetFigFont{12}{14.4}{\rmdefault}{\mddefault}{\updefault}{\color[rgb]{0,0,0}$\delta$}%
}}}}
\put(3451,-1111){\makebox(0,0)[lb]{\smash{{\SetFigFont{10}{12.0}{\rmdefault}{\mddefault}{\updefault}{\color[rgb]{0,0,0}$q$}%
}}}}
\put(2551,-1111){\makebox(0,0)[lb]{\smash{{\SetFigFont{12}{14.4}{\rmdefault}{\mddefault}{\updefault}{\color[rgb]{0,0,0}$\delta$}%
}}}}
\put(1951,-1711){\makebox(0,0)[lb]{\smash{{\SetFigFont{12}{14.4}{\rmdefault}{\mddefault}{\updefault}{\color[rgb]{0,0,0}$\delta$}%
}}}}
\put(1951,-1111){\makebox(0,0)[lb]{\smash{{\SetFigFont{10}{12.0}{\rmdefault}{\mddefault}{\updefault}{\color[rgb]{0,0,0}$q$}%
}}}}
\put(451,-1111){\makebox(0,0)[lb]{\smash{{\SetFigFont{10}{12.0}{\rmdefault}{\mddefault}{\updefault}{\color[rgb]{0,0,0}$u$}%
}}}}
\put(4951,-1111){\makebox(0,0)[lb]{\smash{{\SetFigFont{10}{12.0}{\rmdefault}{\mddefault}{\updefault}{\color[rgb]{0,0,0}$u$}%
}}}}
\put(4051,-1111){\makebox(0,0)[lb]{\smash{{\SetFigFont{12}{14.4}{\rmdefault}{\mddefault}{\updefault}{\color[rgb]{0,0,0}$\delta$}%
}}}}
\end{picture}%

%% file: Basefigure.pstex_t
\begin{picture}(0,0)%
\includegraphics{Basefigure.pstex}%
\end{picture}%
\setlength{\unitlength}{1066sp}%
\begingroup\makeatletter\ifx\SetFigFont\undefined%
\gdef\SetFigFont#1#2#3#4#5{%
  \reset@font\fontsize{#1}{#2pt}%
  \fontfamily{#3}\fontseries{#4}\fontshape{#5}%
  \selectfont}%
\fi\endgroup%
\begin{picture}(15344,9701)(279,-9440)
\put(751,-3211){\makebox(0,0)[lb]{\smash{{\SetFigFont{9}{10.8}{\rmdefault}{\mddefault}{\updefault}{\color[rgb]{0,0,0}$\delta$}%
}}}}
\put(751,-3811){\makebox(0,0)[lb]{\smash{{\SetFigFont{9}{10.8}{\rmdefault}{\mddefault}{\updefault}{\color[rgb]{0,0,0}$\beta$}%
}}}}
\put(8551,-3211){\makebox(0,0)[lb]{\smash{{\SetFigFont{9}{10.8}{\rmdefault}{\mddefault}{\updefault}{\color[rgb]{0,0,0}$\delta$}%
}}}}
\put(8551,-3811){\makebox(0,0)[lb]{\smash{{\SetFigFont{9}{10.8}{\rmdefault}{\mddefault}{\updefault}{\color[rgb]{0,0,0}$\beta$}%
}}}}
\put(9751,-2611){\makebox(0,0)[lb]{\smash{{\SetFigFont{9}{10.8}{\rmdefault}{\mddefault}{\updefault}{\color[rgb]{0,0,0}$q$}%
}}}}
\put(8551,-2611){\makebox(0,0)[lb]{\smash{{\SetFigFont{9}{10.8}{\rmdefault}{\mddefault}{\updefault}{\color[rgb]{0,0,0}$u$}%
}}}}
\put(9751,-3211){\makebox(0,0)[lb]{\smash{{\SetFigFont{9}{10.8}{\rmdefault}{\mddefault}{\updefault}{\color[rgb]{0,0,0}$\beta$}%
}}}}
\put(9751,-3811){\makebox(0,0)[lb]{\smash{{\SetFigFont{9}{10.8}{\rmdefault}{\mddefault}{\updefault}{\color[rgb]{0,0,0}$\delta$}%
}}}}
\put(3451,-3211){\makebox(0,0)[lb]{\smash{{\SetFigFont{9}{10.8}{\rmdefault}{\mddefault}{\updefault}{\color[rgb]{0,0,0}$\delta$}%
}}}}
\put(3451,-3811){\makebox(0,0)[lb]{\smash{{\SetFigFont{9}{10.8}{\rmdefault}{\mddefault}{\updefault}{\color[rgb]{0,0,0}$\beta$}%
}}}}
\put(7051,-2611){\makebox(0,0)[lb]{\smash{{\SetFigFont{9}{10.8}{\rmdefault}{\mddefault}{\updefault}{\color[rgb]{0,0,0}$q$}%
}}}}
\put(7051,-3211){\makebox(0,0)[lb]{\smash{{\SetFigFont{9}{10.8}{\rmdefault}{\mddefault}{\updefault}{\color[rgb]{0,0,0}$\beta$}%
}}}}
\put(7051,-3811){\makebox(0,0)[lb]{\smash{{\SetFigFont{9}{10.8}{\rmdefault}{\mddefault}{\updefault}{\color[rgb]{0,0,0}$\delta$}%
}}}}
\put(5851,-3211){\makebox(0,0)[lb]{\smash{{\SetFigFont{9}{10.8}{\rmdefault}{\mddefault}{\updefault}{\color[rgb]{0,0,0}$\beta$}%
}}}}
\put(5851,-3811){\makebox(0,0)[lb]{\smash{{\SetFigFont{9}{10.8}{\rmdefault}{\mddefault}{\updefault}{\color[rgb]{0,0,0}$\delta$}%
}}}}
\put(4651,-2611){\makebox(0,0)[lb]{\smash{{\SetFigFont{9}{10.8}{\rmdefault}{\mddefault}{\updefault}{\color[rgb]{0,0,0}$u$}%
}}}}
\put(4651,-3211){\makebox(0,0)[lb]{\smash{{\SetFigFont{9}{10.8}{\rmdefault}{\mddefault}{\updefault}{\color[rgb]{0,0,0}$\delta$}%
}}}}
\put(4651,-3811){\makebox(0,0)[lb]{\smash{{\SetFigFont{9}{10.8}{\rmdefault}{\mddefault}{\updefault}{\color[rgb]{0,0,0}$\beta$}%
}}}}
\put(1951,-3211){\makebox(0,0)[lb]{\smash{{\SetFigFont{9}{10.8}{\rmdefault}{\mddefault}{\updefault}{\color[rgb]{0,0,0}$\beta$}%
}}}}
\put(1951,-3811){\makebox(0,0)[lb]{\smash{{\SetFigFont{9}{10.8}{\rmdefault}{\mddefault}{\updefault}{\color[rgb]{0,0,0}$\delta$}%
}}}}
\put(3376,-2611){\makebox(0,0)[lb]{\smash{{\SetFigFont{9}{10.8}{\rmdefault}{\mddefault}{\updefault}{\color[rgb]{0,0,0}$G$}%
}}}}
\put(5776,-2611){\makebox(0,0)[lb]{\smash{{\SetFigFont{9}{10.8}{\rmdefault}{\mddefault}{\updefault}{\color[rgb]{0,0,0}$G$}%
}}}}
\put(1351,-6961){\makebox(0,0)[lb]{\smash{{\SetFigFont{9}{10.8}{\rmdefault}{\mddefault}{\updefault}{\color[rgb]{0,0,0}$q$}%
}}}}
\put(12451,-9211){\makebox(0,0)[lb]{\smash{{\SetFigFont{9}{10.8}{\rmdefault}{\mddefault}{\updefault}{\color[rgb]{0,0,0}$\beta$}%
}}}}
\put(12451,-8611){\makebox(0,0)[lb]{\smash{{\SetFigFont{9}{10.8}{\rmdefault}{\mddefault}{\updefault}{\color[rgb]{0,0,0}$q$}%
}}}}
\put(12451,-7111){\makebox(0,0)[lb]{\smash{{\SetFigFont{9}{10.8}{\rmdefault}{\mddefault}{\updefault}{\color[rgb]{0,0,0}$\delta$}%
}}}}
\put(12376,-6511){\makebox(0,0)[lb]{\smash{{\SetFigFont{9}{10.8}{\rmdefault}{\mddefault}{\updefault}{\color[rgb]{0,0,0}$G$}%
}}}}
\put(12451,-7861){\makebox(0,0)[lb]{\smash{{\SetFigFont{9}{10.8}{\rmdefault}{\mddefault}{\updefault}{\color[rgb]{0,0,0}$q^*$}%
}}}}
\put(11251,-9286){\makebox(0,0)[lb]{\smash{{\SetFigFont{9}{10.8}{\rmdefault}{\mddefault}{\updefault}{\color[rgb]{0,0,0}$\beta$}%
}}}}
\put(11251,-8686){\makebox(0,0)[lb]{\smash{{\SetFigFont{9}{10.8}{\rmdefault}{\mddefault}{\updefault}{\color[rgb]{0,0,0}$\delta$}%
}}}}
\put(11251,-8086){\makebox(0,0)[lb]{\smash{{\SetFigFont{9}{10.8}{\rmdefault}{\mddefault}{\updefault}{\color[rgb]{0,0,0}$u$}%
}}}}
\put(13951,-8611){\makebox(0,0)[lb]{\smash{{\SetFigFont{9}{10.8}{\rmdefault}{\mddefault}{\updefault}{\color[rgb]{0,0,0}$\delta$}%
}}}}
\put(13951,-9211){\makebox(0,0)[lb]{\smash{{\SetFigFont{9}{10.8}{\rmdefault}{\mddefault}{\updefault}{\color[rgb]{0,0,0}$\beta$}%
}}}}
\put(15151,-8011){\makebox(0,0)[lb]{\smash{{\SetFigFont{9}{10.8}{\rmdefault}{\mddefault}{\updefault}{\color[rgb]{0,0,0}$q$}%
}}}}
\put(13951,-8011){\makebox(0,0)[lb]{\smash{{\SetFigFont{9}{10.8}{\rmdefault}{\mddefault}{\updefault}{\color[rgb]{0,0,0}$u$}%
}}}}
\put(15151,-8611){\makebox(0,0)[lb]{\smash{{\SetFigFont{9}{10.8}{\rmdefault}{\mddefault}{\updefault}{\color[rgb]{0,0,0}$\beta$}%
}}}}
\put(15151,-9211){\makebox(0,0)[lb]{\smash{{\SetFigFont{9}{10.8}{\rmdefault}{\mddefault}{\updefault}{\color[rgb]{0,0,0}$\delta$}%
}}}}
\put(9751,-9211){\makebox(0,0)[lb]{\smash{{\SetFigFont{9}{10.8}{\rmdefault}{\mddefault}{\updefault}{\color[rgb]{0,0,0}$\beta$}%
}}}}
\put(9751,-8611){\makebox(0,0)[lb]{\smash{{\SetFigFont{9}{10.8}{\rmdefault}{\mddefault}{\updefault}{\color[rgb]{0,0,0}$q$}%
}}}}
\put(9751,-7861){\makebox(0,0)[lb]{\smash{{\SetFigFont{9}{10.8}{\rmdefault}{\mddefault}{\updefault}{\color[rgb]{0,0,0}$q^*$}%
}}}}
\put(9751,-7111){\makebox(0,0)[lb]{\smash{{\SetFigFont{9}{10.8}{\rmdefault}{\mddefault}{\updefault}{\color[rgb]{0,0,0}$\delta$}%
}}}}
\put(8551,-9211){\makebox(0,0)[lb]{\smash{{\SetFigFont{9}{10.8}{\rmdefault}{\mddefault}{\updefault}{\color[rgb]{0,0,0}$\beta$}%
}}}}
\put(8551,-8611){\makebox(0,0)[lb]{\smash{{\SetFigFont{9}{10.8}{\rmdefault}{\mddefault}{\updefault}{\color[rgb]{0,0,0}$q$}%
}}}}
\put(8551,-7861){\makebox(0,0)[lb]{\smash{{\SetFigFont{9}{10.8}{\rmdefault}{\mddefault}{\updefault}{\color[rgb]{0,0,0}$q^*$}%
}}}}
\put(8551,-7111){\makebox(0,0)[lb]{\smash{{\SetFigFont{9}{10.8}{\rmdefault}{\mddefault}{\updefault}{\color[rgb]{0,0,0}$\delta$}%
}}}}
\put(8551,-6511){\makebox(0,0)[lb]{\smash{{\SetFigFont{9}{10.8}{\rmdefault}{\mddefault}{\updefault}{\color[rgb]{0,0,0}$u$}%
}}}}
\put(7351,-9211){\makebox(0,0)[lb]{\smash{{\SetFigFont{9}{10.8}{\rmdefault}{\mddefault}{\updefault}{\color[rgb]{0,0,0}$\beta$}%
}}}}
\put(7351,-8611){\makebox(0,0)[lb]{\smash{{\SetFigFont{9}{10.8}{\rmdefault}{\mddefault}{\updefault}{\color[rgb]{0,0,0}$\delta$}%
}}}}
\put(7351,-8011){\makebox(0,0)[lb]{\smash{{\SetFigFont{9}{10.8}{\rmdefault}{\mddefault}{\updefault}{\color[rgb]{0,0,0}$u$}%
}}}}
\put(6151,-9211){\makebox(0,0)[lb]{\smash{{\SetFigFont{9}{10.8}{\rmdefault}{\mddefault}{\updefault}{\color[rgb]{0,0,0}$\beta$}%
}}}}
\put(6151,-8611){\makebox(0,0)[lb]{\smash{{\SetFigFont{9}{10.8}{\rmdefault}{\mddefault}{\updefault}{\color[rgb]{0,0,0}$q$}%
}}}}
\put(6151,-7861){\makebox(0,0)[lb]{\smash{{\SetFigFont{9}{10.8}{\rmdefault}{\mddefault}{\updefault}{\color[rgb]{0,0,0}$q^*$}%
}}}}
\put(6151,-7111){\makebox(0,0)[lb]{\smash{{\SetFigFont{9}{10.8}{\rmdefault}{\mddefault}{\updefault}{\color[rgb]{0,0,0}$\delta$}%
}}}}
\put(6151,-6511){\makebox(0,0)[lb]{\smash{{\SetFigFont{9}{10.8}{\rmdefault}{\mddefault}{\updefault}{\color[rgb]{0,0,0}$u$}%
}}}}
\put(4651,-9211){\makebox(0,0)[lb]{\smash{{\SetFigFont{9}{10.8}{\rmdefault}{\mddefault}{\updefault}{\color[rgb]{0,0,0}$\beta$}%
}}}}
\put(4651,-8611){\makebox(0,0)[lb]{\smash{{\SetFigFont{9}{10.8}{\rmdefault}{\mddefault}{\updefault}{\color[rgb]{0,0,0}$q$}%
}}}}
\put(4651,-7861){\makebox(0,0)[lb]{\smash{{\SetFigFont{9}{10.8}{\rmdefault}{\mddefault}{\updefault}{\color[rgb]{0,0,0}$q^*$}%
}}}}
\put(4651,-7111){\makebox(0,0)[lb]{\smash{{\SetFigFont{9}{10.8}{\rmdefault}{\mddefault}{\updefault}{\color[rgb]{0,0,0}$\delta$}%
}}}}
\put(3451,-9211){\makebox(0,0)[lb]{\smash{{\SetFigFont{9}{10.8}{\rmdefault}{\mddefault}{\updefault}{\color[rgb]{0,0,0}$\beta$}%
}}}}
\put(3451,-8461){\makebox(0,0)[lb]{\smash{{\SetFigFont{9}{10.8}{\rmdefault}{\mddefault}{\updefault}{\color[rgb]{0,0,0}$q^*$}%
}}}}
\put(3451,-7711){\makebox(0,0)[lb]{\smash{{\SetFigFont{9}{10.8}{\rmdefault}{\mddefault}{\updefault}{\color[rgb]{0,0,0}$\delta$}%
}}}}
\put(3451,-7111){\makebox(0,0)[lb]{\smash{{\SetFigFont{9}{10.8}{\rmdefault}{\mddefault}{\updefault}{\color[rgb]{0,0,0}$u$}%
}}}}
\put(1951,-8611){\makebox(0,0)[lb]{\smash{{\SetFigFont{9}{10.8}{\rmdefault}{\mddefault}{\updefault}{\color[rgb]{0,0,0}$\beta$}%
}}}}
\put(1951,-7861){\makebox(0,0)[lb]{\smash{{\SetFigFont{9}{10.8}{\rmdefault}{\mddefault}{\updefault}{\color[rgb]{0,0,0}$q^*$}%
}}}}
\put(1951,-7111){\makebox(0,0)[lb]{\smash{{\SetFigFont{9}{10.8}{\rmdefault}{\mddefault}{\updefault}{\color[rgb]{0,0,0}$\delta$}%
}}}}
\put(451,-8611){\makebox(0,0)[lb]{\smash{{\SetFigFont{9}{10.8}{\rmdefault}{\mddefault}{\updefault}{\color[rgb]{0,0,0}$\beta$}%
}}}}
\put(451,-7861){\makebox(0,0)[lb]{\smash{{\SetFigFont{9}{10.8}{\rmdefault}{\mddefault}{\updefault}{\color[rgb]{0,0,0}$q^*$}%
}}}}
\put(451,-7111){\makebox(0,0)[lb]{\smash{{\SetFigFont{9}{10.8}{\rmdefault}{\mddefault}{\updefault}{\color[rgb]{0,0,0}$\delta$}%
}}}}
\put(9676,-6511){\makebox(0,0)[lb]{\smash{{\SetFigFont{9}{10.8}{\rmdefault}{\mddefault}{\updefault}{\color[rgb]{0,0,0}$G$}%
}}}}
\end{picture}%